\documentclass[reqno,a4paper]{amsart}

\usepackage[english]{babel}
\usepackage{amsmath, amsfonts, amssymb, amsthm}
\usepackage{mathdots}

\usepackage[a4paper]{geometry}

\usepackage{pgf,tikz}
\usetikzlibrary{arrows}

\DeclareMathOperator{\tg}{tan}
\DeclareMathOperator{\ctg}{cot}
\DeclareMathOperator{\sh}{sinh}
\DeclareMathOperator{\ch}{cosh}
\DeclareMathOperator{\cth}{coth}

\DeclareMathOperator{\arcctg}{arccot}
\DeclareMathOperator{\arcch}{arcosh}
\DeclareMathOperator{\arccth}{arcoth}

\DeclareMathOperator{\grad}{grad}

\theoremstyle{plain}
\newtheorem{theorem}{Theorem}
\newtheorem{lemma}{Lemma}[section]

\theoremstyle{remark}
\newtheorem{remark}{Remark}[section]

\numberwithin{equation}{section}

\usepackage{caption}
\DeclareCaptionLabelSeparator{dot}{. }
\title[Closeness to spheres of hypersurfaces]{Closeness to spheres of hypersurfaces with normal curvature bounded below}

\author[A.~Borisenko]{Alexander Borisenko}
\address[Alexander Borisenko]{Mathematics Department \\ 
Sumy State University \\
R.~-- Korsakova Street 2, 40007 \\ Sumy \\ Ukraine}
\email{borisenk@univer.kharkov.ua, aborisenk@gmail.com}

\author[K.~Drach]{Kostiantyn Drach}
\address[Kostiantyn Drach]{Geometry Department \\ 
The School of Mechanical Engineering and Mathematics\\
V.~N.~Karazin Kharkiv National University \\
Svobody Sq. 4, 61022 \\ Ukraine}
\email{kostya.drach@gmail.com}

\begin{document}

\begin{abstract}
For a Riemannian manifold $M^{n+1}$ and a compact domain
$\Omega \subset M^{n+1}$ bounded by a hypersurface $\partial \Omega$ with normal curvature bounded
below, estimates are obtained in terms of the distance from $O$ to $\partial \Omega$ for
the angle between the geodesic line joining a fixed interior point $O$ in $\Omega$ to
a point on $\partial \Omega$ and the outward normal to the surface. Estimates for the
width of a spherical shell containing such a hypersurface are also presented.
\end{abstract}


\maketitle

\section*{Introduction}

W.~Blaschke has proved~\cite{Bla} that if normal curvatures of a complete hypersurface $F^n$ in the Euclidean space are pinched between some positive constants $k_1$ and $k_2$, i.~e. $k_2 \geqslant k_n \geqslant k_1 >0$, then at every point on $F^n$ there exist supporting spheres with the radii $1/k_1$ and $1/k_2$ that encloses and is enclosed by $F^n$ respectively. But it appears that we can say something about the sphericity of a surface even if the normal curvature is bounded only from below.

Let us consider an arbitrary circle on the Euclidean plane. Obviously, the angle between a ray from the center through any point on the circle and the outer normal at this point is identically zero. If to consider rays emanating not from the center but from another fixed point $O$ inside the circle, the similar angle will not longer be identically zero. The same holds for arbitrary convex curves on the plane. However, the closer all of these angles are to zero, the closer a curve is to a circle and the point $O$~-- to its center. Thus, we can see that the values of the considered angles reveal the closeness of a curve to a circle. This is a motivation for us to study such angles in more general settings.

 Notably, if a hypersurface has the normal curvature bounded from below, then for points $O$ inside the domain enclosed by the surface at the distance $h$ from it the angles between normals and radial directions from $O$ cannot be to big. In~\cite{BorMiq1} estimates for such angles were obtained for surfaces in ${{\mathbb H}}^n(-1)$~-- Lobachevsky space of the constant negative curvature $-1$, provided that all normal curvatures of the surface $k_n \geqslant 1$ or $k_n\geqslant\lambda $, $\lambda <1$. These estimates were generalized in~\cite{BorGalRev} and~\cite{Bor} for hypersurfaces lying in the Hadamard manifolds~-- complete simply connected Riemannian manifold of the negative sectional curvature $K$ satisfying $-k^2_1\leqslant K\leqslant 0$ with some positive constant $k_1$, provided that all normal curvatures of the surface $k_n \geqslant k_1$ or $k_n \geqslant \lambda$, $\lambda < k_1$. 

Thereby, the open question was whether similar estimates can be obtained for surfaces in Riemannian manifolds of the non-positive sectional curvature $0\geqslant K \geqslant-k^2_1, k_1\geqslant0$ if all normal curvatures of a surface $k_n\geqslant\lambda$, $\lambda>k_1$ and in manifolds of the positive sectional curvature, provided that $k_n\geqslant\lambda \geqslant0$. In the two-dimensional case such estimates were announced in~\cite{BorDr}. These results alongside with their multidimensional generalizations make the content of the first part of our paper.     

Another way to  measure the sphericity of hypersurfaces is to consider the width of a spherical layer which can enclose a hypersurface. It is quite clear, that the smaller this width is, the closer our surface is to a sphere.  
 
In the paper~\cite{BorMiq1} it was proved, that a closed hypersurface in $\mathbb H^n(-1)$ can be put into the spherical layer of the width $d \leqslant \ln 2$, provided that normal curvatures of the surface $k_n\geqslant 1$ at any point and in any direction. A similar estimate holds in the Hadamard manifolds (see~\cite{BorMiq2}). In the second part of our paper we extend these results for manifolds of the constant sectional curvature and to the general Riemannian case with another bounds for the normal curvature of a hypersurface.

\section{Preliminaries and statements of the main results}

Let us consider a complete simply connected $(n+1)$-dimensional Riemannian manifold $M^{n+1}$. We will denote the sectional curvature of $M^{n+1}$ at an arbitrary point $P\in M^{n+1}$ in the direction of a two-dimensional plane $\sigma \subset T_PM^{n+1}$ as $K_{\sigma}$. Let $\Omega \subset M^{n+1}$ be a  closed  compact domain whose boundary $\partial\Omega$ is a $C^2$-smooth hypersurface.   

Consider a point $O\in\Omega$. Let $h:=dist(O,\partial\Omega)$ be the distance from this point to the boundary of the domain. Denote $\varphi:=\varphi(P)$ to be an angle between the geodesic line from the point $O$ to an arbitrary point $P\in\partial\Omega$ and the outer normal to $\partial\Omega$ in the point $P$ (see Fig.~\ref{ch1pic1}).

\begin{figure}[h]
\begin{center}
\definecolor{uququq}{rgb}{0.25,0.25,0.25}
\begin{tikzpicture}[line cap=round,line join=round,>=triangle 45,x=1.0cm,y=1.0cm,scale=2.15]
\clip(-2.3,-2.1) rectangle (1.9,2.2);
\fill[line width=0.4pt,dotted,fill=black,fill opacity=0.05] (-0.06,2.04) -- (1.6,1.78) -- (1.33,0.02) -- (-0.33,0.28) -- cycle;
\draw [shift={(0.64,1.03)},line width=0.4pt,fill=black,fill opacity=0.05] (0,0) -- (27.92:0.18) arc (27.92:81.15:0.18) -- cycle;
\draw [rotate around={38.35:(-0.32,-0.23)},line width=1.2pt] (-0.32,-0.23) ellipse (1.98cm and 1.45cm);
\draw [shift={(-0.52,-0.77)}] plot[domain=3.16:6.03,variable=\t]({1*1.59*cos(\t r)+0.01*0.56*sin(\t r)},{-0.01*1.59*cos(\t r)+1*0.56*sin(\t r)});
\draw [shift={(-0.52,-0.77)},dash pattern=on 1pt off 1pt]  plot[domain=-0.25:3.16,variable=\t]({1*1.59*cos(\t r)+0.01*0.56*sin(\t r)},{-0.01*1.59*cos(\t r)+1*0.56*sin(\t r)});
\draw [dash pattern=on 1pt off 1pt] (-1.03,0.73)-- (-0.64,0.36);
\draw [shift={(-0.43,0.22)},line width=0.4pt,dash pattern=on 1pt off 1pt]  plot[domain=0.23:1.82,variable=\t]({0.95*1.65*cos(\t r)+-0.31*0.69*sin(\t r)},{0.31*1.65*cos(\t r)+0.95*0.69*sin(\t r)});
\draw [dash pattern=on 1pt off 1pt] (-0.64,0.36)-- (0.64,1.03);
\draw [->] (0.64,1.03) -- (0.75,1.74);
\draw [->] (0.64,1.03) -- (1.27,1.37);
\draw [->] (0.64,1.03) -- (1.34,0.92);
\draw [->] (0.64,1.03) -- (0.97,0.4);
\draw [line width=0.4pt,dotted] (-0.06,2.04)-- (1.6,1.78);
\draw [line width=0.4pt,dotted] (1.6,1.78)-- (1.33,0.02);
\draw [line width=0.4pt,dotted] (1.33,0.02)-- (-0.33,0.28);
\draw [line width=0.4pt,dotted] (-0.33,0.28)-- (-0.06,2.04);
\draw [line width=0.4pt,fill=black,fill opacity=0.05] (-0.64,0.36) circle (1.44cm);
\draw [shift={(-0.64,0.36)},line width=0.4pt,dotted]  plot[domain=0:pi,variable=\t]({1*1.44*cos(\t r)+0*0.43*sin(\t r)},{0*1.44*cos(\t r)+1*0.43*sin(\t r)});
\draw [shift={(-0.64,0.36)},line width=0.4pt]  plot[domain=-3.14:0,variable=\t]({1*1.44*cos(\t r)+0*0.43*sin(\t r)},{0*1.44*cos(\t r)+1*0.43*sin(\t r)});
\draw (-1.1,-1.33) node[anchor=north west] {$\Omega$};
\draw (1.58,0.08) node[anchor=north west] {$\partial\Omega$};
\draw (-0.47,0.58) node[anchor=north west] {$t=\rho\left(\theta^1,...,\theta^n\right)$};
\draw (1.08,1.7) node[anchor=north west] {$\partial_t$};
\fill [color=black] (-0.64,0.36) circle (1.0pt);
\draw[color=black] (-0.72,0.22) node {$O$};
\fill [color=black] (-1.03,0.73) circle (1.0pt);
\draw[color=black] (-1.07,0.86) node {$Q_0$};
\draw[color=black] (-0.91,0.52) node {$h$};
\draw[color=black] (-0.3,1.06) node {$\gamma$};
\fill [color=black] (0.64,1.03) circle (1.0pt);
\draw[color=black] (0.54,1.17) node {$P$};
\draw[color=black] (0.61,1.79) node {$n$};
\draw[color=black] (1.34,1.1) node {$Y$};
\draw[color=black] (1.03,0.58) node {$X$};
\draw[color=black] (0.81,1.23) node {$\varphi$};
\end{tikzpicture}
\caption{}
\label{ch1pic1}
\end{center}
\end{figure}

Hereafter, we will use a notation $k_n := k_n(P, Y)$ for the normal curvature of the hypersurface $\partial\Omega$ at a point $P\in\partial\Omega$ in the direction of a vector $Y\in T_P\partial\Omega$.

It appears that, if all normal curvatures of $\partial\Omega$ in any direction are bounded from below $k_n \geqslant k_0$, then the angle $\varphi$ cannot be to big. Namely, the following theorems hold.

\begin{theorem}
\label{ch1th1}
Let $M^{n+1}(c)$ be a complete simply connected Riemannian manifold of the constant sectional curvature $c$, $\Omega$ be a domain in it whose boundary $\partial\Omega$ is a $C^2$-smooth hypersurface. Let $O\in\Omega$ be a point inside the domain, $h=dist(O,\partial\Omega)$ be the distance from $O$ to the hypersurface and $\varphi$ be the angle between a radial direction from the point $O$ to a point on $\partial\Omega$ and the outer normal taken at this point.
\begin{enumerate}
\item 
If $c=0$, i.e. $M^{n+1}(c)=\mathbb E ^{n+1}$ is the Euclidean space, and all normal curvatures of $\partial\Omega$ in any direction $k_n\geqslant k_0 > 0$, then
\begin{equation}
\label{ch1th1c1}
\cos \varphi \geqslant \sqrt{2 h k_0 - h^2 k_0^2} \geqslant h k_0.
\end{equation}

\item 
If $c=-k_1^2, k_1>0$, i.e. $M^{n+1}(c)=\mathbb H^{n+1}(-k_1^2)$ is the $(n+1)$-dimensional Lobachevsky space, and all normal curvatures of $\partial\Omega$ in any direction $k_n\geqslant k_0 > k_1$, then 
\begin{equation}
\label{ch1th1c2}
\cos \varphi \geqslant \sqrt{1-\frac{\sh^2 k_1 (R-h)}{\sh^2 k_1 R}}\geqslant \frac{\sh k_1 h}{\sh k_1 R},
\end{equation}
where $R=\frac{1}{k_1} \arccth \frac{k_0}{k_1}$ is the radius of a circle of the curvature $k_0$ on the two-dimensional Lobachevsky plane of the Gaussian curvature $-k_1^2$.

\item
If $c=k_1^2, k_1>0$, i.e. $M^{n+1}(c)=S^{n+1}(k_1^2)$ is the $(n+1)$-dimensional sphere, and all normal curvatures of $\partial\Omega$ in any direction $k_n\geqslant k_0 \geqslant 0$, then 
\begin{equation}
\label{ch1th1c3}
\cos \varphi \geqslant \sqrt{1-\frac{\sin^2 k_1 (R-h)}{\sin^2 k_1 R}} \geqslant \frac{\sin k_1 h}{\sin k_1 R},
\end{equation}
where $R=\frac{1}{k_1} \arcctg \frac{k_0}{k_1}$ is the radius of a circle of the curvature $k_0$ on the two-dimensional sphere of the Gaussian curvature $k_1^2$.
\end{enumerate}
\end{theorem}

The similar result holds if the ambient space is a Riemannian manifold of the constant-sign sectional curvature.

\begin{theorem}
\label{ch1th2}
Let $\partial\Omega$ be a complete $C^2$-smooth hypersurface in a complete simply connected $(n+1)$-dimensional Riemannian manifold $M^{n+1}$, $\Omega \subset M^{n+1}$ is the domain bounded by $\partial\Omega$.  
\begin{enumerate}
\item If all sectional curvatures $K_\sigma$ of $M^{n+1}$ with respect to any two-\-di\-men\-si\-onal plane $\sigma$ satisfy the inequality $0\geqslant K_\sigma \geqslant -k_1^2$, $k_1 > 0$, and all normal curvatures of $\partial\Omega$ in any direction $k_n \geqslant k_0 > k_1$, then the estimate~(\ref{ch1th1c2}) holds.
\item If all sectional curvatures of the manifold $M^{n+1}$ satisfy $k_2^2 \geqslant K_\sigma \geqslant k_1^2$, $k_1 > 0$, and the domain $\Omega$ lies in a ball with the radius $\pi/2k_2$, then if all normal curvatures of $\partial\Omega$ in any direction $k_n \geqslant k_0 \geqslant 0$, the estimate~(\ref{ch1th1c3}) holds.
\end{enumerate}
\end{theorem}

Let us recall (see~\cite{Bor},~\cite{ BorGalRev}) that a locally convex hypersurface $\partial \Omega \subset M^{n+1}(c)$ is  \textit{$\lambda$-convex} if at every point $P \in \partial\Omega$ there is a sphere $S_P$ of the sectional curvature $\lambda^2$ passing through this point such that in the neighborhood of $P$ the hypersurface lies on the convex side of $S_P$. The corresponding domain $\Omega$ is called a \textit{$\lambda$-convex} domain. Note that $\partial\Omega$ can be non-regular.

We note that regular with the class $C^k$, $k\geqslant 2$, hypersurface $\partial\Omega$ is $\lambda$-convex if and only if all its normal curvatures at any point and in any direction satisfy $k_n \geqslant \lambda$. Thereby, the notion of $\lambda$-convexity  is the non-regular generalization of the fact that normal curvatures are bounded from below by $\lambda$. 

Taking into account all the definitions above, for the width of a spherical layer the following theorem holds.

\begin{theorem}
\label{ch2th1}
Let $\partial\Omega$ be a complete hypersurface in a complete simply connected $(n+1)$-\-di\-men\-si\-onal Riemannian manifold $M^{n+1}(c)$ of the constant sectional curvature $c$.
\begin{enumerate}
\item 
Suppose that the ambient space is the Euclidean space  $\mathbb E^{n+1}$. If $\partial\Omega$ is a $k_0$-convex hypersurface, $k_0>0$, then $\partial\Omega$ can be enclosed in a spherical layer of the width
\begin{equation}
\label{ch2eq1}
d \leqslant \frac{\sqrt{2} - 1}{k_0}.
\end{equation}

\item 
Suppose that the ambient space $M^{n+1} (c) = S^{n+1} (k_1^2)$~--~$(n+1)$\--di\-men\-si\-onal sphere, $k_1 > 0$. If $\partial\Omega$ is a $k_0$-convex hypersurface, $k_0>0$, then $\partial\Omega$ can be enclosed in a spherical layer of the width      
\begin{equation}
\label{ch2eq2}
d \leqslant \frac{2}{k_1}\arccos \sqrt{\cos k_1 R} - R,
\end{equation}
where $R$ is the radius of a circle of the curvature $k_0$ on the two\--di\-men\-si\-onal sphere of the Gaussian curvature $k_1^2$.

\item 
Suppose that the ambient space $M^{n+1} (c) = \mathbb H ^{n+1}(-k_1^2)$, $k_1>0$, is the Lobachevsky space. If $\partial\Omega$ is a $k_0$-convex hypersurface, $k_0>k_1$, then $\partial\Omega$ lies in the spherical layer of the width 
\begin{equation}
\label{ch2eq3}
d \leqslant \frac{2}{k_1}\arcch\sqrt{\ch k_1 R} - R
\end{equation}
where $R$ is the radius of a circle of the curvature $k_0$ on the two\--di\-men\-si\-onal Lobachevsky plane of the Gaussian curvature $-k_1^2$.
\end{enumerate}

\end{theorem}

\begin{remark}
Hereafter, we will fix the notation $R$ for the radius of a circle of the curvature $k_0$ on the plane of the constant curvature $c$. Taking into account the values for $R$ stated in Theorem~\ref{ch1th1}, the estimates~(\ref{ch2eq2}),~(\ref{ch2eq3}) can be rewritten, accordingly:
\begin{itemize}
\item[2.] 
$d \leqslant \frac{1}{k_1} \left( 2 \arccos \frac{\sqrt{k_0}}{\sqrt[4]{k_0^2 + k_1^2}} - \arcctg \frac{k_0}{k_1} \right)$;

\item[3.] 
$d \leqslant \frac{1}{k_1} \left( 2 \arcch \frac{\sqrt{k_0}}{\sqrt[4]{k_0^2 + k_1^2}} - \arccth \frac{k_0}{k_1} \right)$.
\end{itemize}
\end{remark}

\begin{remark}
Given estimates are sharp in the meaning that there are examples of hypersurfaces for which the minimal width of a spherical layer that encloses the hypersurface has the value equal to the right-hand member of the inequalities in Theorem~\ref{ch2th1}.  
\end{remark}

In the following theorem we generalize the estimates for the width of a spherical layer to the ambient manifolds of the constant-sign sectional curvature.

\begin{theorem}
\label{genth1}
Let $\partial\Omega$ be a complete $C^2$-smooth hypersurface in the complete simply connected $(n+1)$-dimensional Riemannian manifold $M^{n+1}$.

\begin{enumerate}
\item 
Let all sectional curvatures of $M^{n+1}$ with respect to any two-dimensional plane $\sigma$, $k_2^2 \geqslant K_\sigma \geqslant k_1^2$, $k_1 > 0$. Suppose that the hypersurface lies in a ball with the radius $\pi/k_2$ and the center which coincides with the center of the inscribed ball for $\partial\Omega$. If all normal curvatures of  $\partial\Omega$ in any direction $k_n \geqslant k_0 > 0$, then the hypersurface $\partial\Omega$ can be enclosed in a spherical layer of the width~(\ref{ch2eq2});

\item 
Let for any two-dimensional plane $\sigma$, $0 \geqslant K_\sigma \geqslant -k_1^2$, $k_1 > 0$. If all normal curvatures of $\partial\Omega$ in any direction $k_n \geqslant k_0 > k_1$, then $\partial\Omega$ lies in a spherical layer of the width~(\ref{ch2eq3}). 

\end{enumerate}
\end{theorem}


\section{Proofs of the angle comparison theorems}


In this section we will prove Theorems~\ref{ch1th1} and~\ref{ch1th2} which are by their nature, as it will be clear further, the angle comparison theorems.

\subsection{Auxiliary results}

Let us introduce on the manifold $M^{n+1}$ the polar coordinate system with the origin at the point $O\in\Omega$. In this coordinate system the arc length can be expressed in the form $ds^2 = dt^2 + g_{ij}d\theta^i d\theta^j$, $i,j = 1..n$, where $t$ is a length parameter, $\theta^i$ are angle parameters.

We can assume that the hypersurface $\partial\Omega$ is given by the equation $t=\rho(\theta^1,\ldots,\theta^n)$. This assumption is valid for convex hypersurfaces lying in the domain of regularity of our coordinate system. Then $\partial\Omega$ is the $0$-level set of the function $F(t,\theta^1,\ldots, \theta^n) = t - \rho(\theta^1,\ldots,\theta^n)$.   

For an arbitrary manifold $N$ and a smooth function $f$ on it the vector field of the gradient of this function is the unique vector field $\grad_Nf$ such that for any $v\in TN$
$$
\langle \grad_Nf, v \rangle = v(f).
$$ 
 
Let us denote by $Y$ the unit gradient vector field of the distance function $\rho$ from the point $O$ to points on $\partial\Omega$ defined on $\partial\Omega$:
$$
Y=\frac{\grad_{\partial\Omega}\rho}{|\grad_{\partial\Omega}\rho|}.
$$ 

We recall that the unit outward normal to the hypersurface $\partial\Omega$ can be written as $$n=\frac{\grad_{M^{n+1}}F}{|\grad_{M^{n+1}}F|}.$$The unit vector field $\partial_t:=\frac{\partial}{\partial t}$ defines the radial directions from $O$ to the points on $\partial\Omega$, $\varphi$ is the angle between $n$ and $\partial_t$.
  
It can be shown (see~\cite{BorGalRev},\cite{Bor}), that the vectors $n(P)$, $\partial_t(P)$ and $Y(P)$ are lying in the same two-dimensional plane in $T_PM^{n+1}$. Let $X(P)$ be a unit vector lying in this plane perpendicularly to $\partial_t(P)$ (see Fig.~\ref{ch1pic1}). Let us denote the normal curvature of $\partial\Omega$ at the point $P\in\partial\Omega$ in the direction of the vector $Y=Y(P)$ as $k_n$.

Then the following lemma holds.

\begin{lemma}[\cite{BorGalRev},\cite{Bor}]
\label{ch1lem1}
If $\mu_n$ is the normal curvature of the sphere with the radius $\rho$ and the center $O$ taken at the point $P\in\partial\Omega$ in the direction of $X$; $\frac{d\varphi}{ds}$ is a derivative of $\varphi$ with respect to the arc length parameter of the integral trajectory of the vector field $Y$ taken at the point $P\in\partial\Omega$. Then 
$$
k_n(s) =  \mu_n(s)\cos\varphi - \frac{d\varphi}{ds}.
$$  
\end{lemma}

\begin{remark}
In the same papers \cite{BorGalRev},\cite{Bor} there was shown that if to parameterize the integral trajectory of $Y$ not by the arc length parameter $s$ but with the distance parameter $t$ from $O$ to the curve (locally), then the last formula can be rewritten in the following way:
$$
k_n(t) =  \mu_n(t)\cos\varphi  - \frac{d\varphi}{dt}\sin\varphi.
$$
\end{remark}

\begin{remark}
In the two-dimensional case $\partial\Omega$ will be a closed embedded $C^2$-smooth curve $\gamma$ on a two-dimensional manifold $M^2$ parameterized with the distance parameter $t$ from the origin $O$; $\mu_n(t) = \mu(t)$ will be the geodesic curvature of the circle with the radius $t$ and the center at the origin taken at the point $\gamma(t)$. Then the geodesic curvature $k$ of $\gamma$ will satisfy the equation
$$
k(t) =  \mu(t)\cos\varphi  - \frac{d\varphi}{dt}\sin\varphi.
$$
\end{remark}

To establish the relations between the curvature of a sphere and the curvature of a space we will need the following lemma.
\begin{lemma}[\cite{Pet},\cite{BurZal}]
\label{ch1lem3}
Let us assume that all sectional curvatures of the Riemannian manifold $M^{n+1}$ satisfy one of the following conditions:
\begin{enumerate}
\item $K_{\sigma}\geqslant k_{1}^{2},\ k_{1}>0$ for all two-dimensional planes $\sigma$ and a sphere with the radius $t$ lies in the domain of regularity of the polar coordinate system with the origin at the center of the sphere.

\item $0\geqslant K_{\sigma}\geqslant -k_{1}^{2}, k_{1} \geqslant 0$.
\end{enumerate}

Then all normal curvatures $\mu_n(t)$ of a sphere with the radius $t$ in any direction satisfy the following inequality
$$
\mu_n(t)\leqslant \mu_{0}(t),
$$where $\mu_{0}(t)$ is the geodesic curvature of a circle with the radius $t$ on the plane with the constant Gaussian curvature, respectively, 

1. $k_{1}^{2}$; 2. $-k_{1}^{2}.$
\end{lemma}

We recall that on two-dimensional planes geodesic curvatures $\mu_0(t)$ of circles with the radius $t$ are equal: $\mu_{0}(t)=\frac{1}{t}$ on the Euclidean plane; $\mu_{0}(t)=k_{1}\ctg k_{1}t$ on the sphere of the Gaussian curvature $k_1^2$;  $\mu_{0}(t)=~k_{1}\cth k_{1}t$ on the Lobachevsky plane of the Gaussian curvature $-k_1^2$.

In order to use the comparison lemma stated above we need to study the behavior of the angle $\varphi $ for circles lying on surfaces with the constant Gaussian curvature.

\begin{lemma}[\cite{BorDr}]
\label{ch1lem4}
Let $M^{2}$ be a plane of the constant Gaussian curvature, $\gamma$ is a circle with the radius $R$ on it, $O$ is a point inside the circle at the distance $h$ from it. Then the angle $\varphi$ between a geodesic line from $O$ to a point $\gamma(s)$ on the circle and the outer normal vector to the circle at this point satisfies the inequality:
\begin{enumerate}
 \item In the case of the Euclidean plane
 $$
 \cos \varphi\geqslant\sqrt{\frac{2h}{R}-\frac{h^{2}}{R^{2}}}.
 $$
 \item In the case of the Lobachevsky plane of the curvature $-k_{1}^{2}$
 $$
 \cos\varphi\geqslant\sqrt{1-\frac{\sh^{2} k_{1}(R-h)}{\sh^{2} k_{1}R }}.
 $$
\item In the case of the sphere of the curvature $k_{1}^{2}$, assuming $R \leqslant \frac{\pi}{2k_1}$,
 $$
 \cos\varphi\geqslant\sqrt{1-\frac{\sin^{2} k_1(R-h)}{\sin^{2} k_{1}R}}.
 $$
\end{enumerate}
In all these cases the equality holds only in the directions perpendicular to the geodesic line connecting the center of the circle with the point $O$.
\end{lemma}

\begin{remark}
In all cases there are simpler, but more rough, estimates:
\begin{enumerate}
 \item
 $\cos \varphi\geqslant\sqrt{\frac{2h}{R}-\frac{h^{2}}{R^{2}}}\geqslant\frac{h}{R}$;
 
 \item
$\cos\varphi\geqslant\sqrt{1-\frac{\sh^{2} k_{1}(R-h)}{\sh^{2} k_{1}R }}\geqslant\frac{\sh k_{1}h}{\sh k_{1}R}$;

\item
 $ \cos\varphi\geqslant\sqrt{1-\frac{\sin^{2} k_1(R-h)}{\sin^{2} k_{1}R}}\geqslant\frac{\sin k_1h}{\sin k_1R}$.
\end{enumerate}
\end{remark}

\begin{proof}

We will prove the estimate in the Euclidean case. For the rest of the cases the proofs are absolutely similar with the necessary replacement of the classical formulas with their spherical and hyperbolic analogs.

Denote the center of $\gamma$ as $O_1$, the intersection of the ray $O_1O$ with $\gamma$ as $M$. Since $h$ is the distance, then $OM=h$. Additionally, if $P\in \gamma$ be an arbitrary point, then $\angle O_1PO=\varphi $ (see Fig.~\ref{ch1pic3}).

\begin{figure}[h]
\begin{center}
\definecolor{uququq}{rgb}{0.25,0.25,0.25}
\begin{tikzpicture}[line cap=round,line join=round,>=triangle 45,x=1.0cm,y=1.0cm,scale=1.2]
\clip(-2.7,-2.2) rectangle (1.7,2.1);
\draw [shift={(-1.38,1.45)}] (0,0) -- (-63.91:0.39) arc (-63.91:-38.51:0.39) -- cycle;
\draw [shift={(-1.38,1.45)}] (0,0) -- (116.09:0.39) arc (116.09:141.49:0.39) -- cycle;
\draw [shift={(0.69,-0.19)}] (0,0) -- (141.49:0.39) arc (141.49:180:0.39) -- cycle;
\draw [line width=1.2pt] (-0.58,-0.19) circle (1.83cm);
\draw (-0.58,-0.19)-- (1.26,-0.19);
\draw [->] (-1.38,1.45) -- (-1.67,2.03);
\draw [->] (-1.38,1.45) -- (-1.89,1.85);
\draw (-1.38,1.45)-- (0.69,-0.19);
\draw (-1.38,1.45)-- (-0.58,-0.19);
\draw [shift={(0.69,-0.19)}] (141.49:0.39) arc (141.49:180:0.39);
\draw [shift={(0.69,-0.19)}] (141.49:0.32) arc (141.49:180:0.32);
\draw [->,line width=0.4pt] (-0.58,-0.19) -- (-2.3,-0.82);
\draw (0.63,-0.18) node[anchor=north west] {${O}$};
\draw (-0.6,-0.18) node[anchor=north west] {${O_1}$};
\draw (1.31,-0.18) node[anchor=north west] {${M}$};
\draw (-1.42,1.91) node[anchor=north west] {${P}$};
\draw (-1.18,1.15) node[anchor=north west] {$\varphi$};
\draw (0.0,0.16) node[anchor=north west] {$\alpha$};
\draw (-1.81,-0.63) node[anchor=north west] {$R$};
\draw (-2.58,0.81) node[anchor=north west] {$\gamma$};
\draw (0.86,0.17) node[anchor=north west] {$h$};
\fill [color=uququq] (-0.58,-0.19) circle (1.0pt);
\fill [color=uququq] (1.26,-0.19) circle (1.0pt);
\fill [color=uququq] (0.69,-0.19) circle (1.0pt);
\fill [color=uququq] (-1.38,1.45) circle (1.0pt);
\end{tikzpicture}
\caption{}
\label{ch1pic3}
\end{center}
\end{figure}

If to denote $\angle POM=\alpha$, then from the law of sines applied to the triangle $\Delta O_1OP$ it follows that: $\frac{O_1P}{\sin  \angle O_1OP}=\frac{O_1O}{\sin\angle O_1PO}$  or  $\frac{R}{\sin  \alpha }=\frac{R-h}{\sin\varphi }$. Thus:
$$
\sin\varphi =\frac{R-h}{R}\sin \alpha.
$$
Taking into account that $\frac{R-h}{R}$ is a constant we will obtain that the maximal value of the angle $\varphi$ (and therefore the maximal value of $\sin \varphi$ and the minimal value of $\cos\varphi$) is attained when $\sin  \alpha =1$, i.e. when $\alpha =\frac{\pi }{2}$, $OP\bot O_1M$.

At the same time, if $\varphi_0=\max\varphi$, then $\sin\varphi_0=\frac{R-h}{R}$. Thus,
$$
\cos\varphi_0=\sqrt{1-\left(\frac{R-h}{R}\right)^2}=\sqrt{1-\left(1-\frac{h}{R}\right)^2}=\sqrt{\frac{2h}{R}-\frac{h^2}{R^2}} \geqslant \frac{h}{R}.
$$
(the last inequality holds since $\frac{2h}{R}-\frac{h^2}{R^2}\geqslant \frac{h^2}{R^2}\Leftrightarrow \frac{h}{R}\geqslant \frac{h^2}{R^2}\Leftrightarrow \frac{h}{R}\leqslant 1 \Leftrightarrow h\leqslant R$, which is obviously true).

Hereby, for an arbitrary angle $\varphi $
$$
\cos\varphi \geqslant \cos \varphi_0=\sqrt{\frac{2h}{R}-\frac{h^2}{R^2}}\geqslant \frac{h}{R}
$$
and the equality in this case holds only for $OP\bot O_1M$. This finishes the proof. 

\end{proof}

Finally, at the end of this section let us prove some useful technical lemma.

\begin{lemma}
\label{ch1lem5}
Let $f(x)\in C^1[a,b]$ is a continuously differentiable function and $f(a)=0$, $f(b)<0$. Then among those values for which $f(x)<0$ there is a value $x_0\in (a,b)$ such that $f(x_0)<0$ and $f'(x_0)<0$.

\end{lemma}

\begin{proof}

Assume the contrary. Then for all $x\in (a,b)$ such that $f(x)<0$, the inequality $f'(x)\ge 0$ holds. 

But since $f$ is continuous and $f(a)=0,\ f(b)<0$, there is a segment $[a_1,b_1)\subset [a,b]$ such that: $f(a_1)=0$ and for all $x\in (a_1,b_1)$ $f(x)<0$. By assumption,  for any $x\in (a_1,b_1)$ $f'(x)\ge 0$. It means that $f$ is a nondecreasing function on $(a_1,b_1)$. Hence, $f(x)\ge f(a_1)=0$, which contradicts the choice of the segment $[a_1,b_1)$. The lemma is proved.  

\end{proof}

\subsection{Proofs of Theorems~\ref{ch1th1} and~\ref{ch1th2}}

We start this section with proving the following two-dimensional 

\begin{lemma}
\label{ch1prooflem1}

Let $\gamma$ be a regular with the class $C^k$, $k\geqslant 2$, closed embedded curve on a plane of the constant Gaussian curvature. Let $O$ be a point inside the domain bounded by $\gamma$ at the distance $h$ from the curve; $\varphi$ is the angle between a radial direction from the point $O$ to a point on $\gamma$ and the outer normal to the curve at this point.

\begin{enumerate}
\item 
If on the Euclidean plane the curvature of the curve satisfies $k\geqslant k_0>0$, then the inequality~(\ref{ch1th1c1}) holds.

\item 
If on the Lobachevsky plane of the curvature $-k_1^2$ ($k_1>0$) the geodesic curvature of $\gamma$ satisfies $k \geqslant k_0>k_1$, then the inequality~(\ref{ch1th1c2}) holds.

\item
If on the sphere of the curvature $k_1^2$ ($k_1>0$) the geodesic curvature of $\gamma$ satisfies $k\geqslant k_0\geqslant 0$, then the inequality~(\ref{ch1th1c3}) holds.

\end{enumerate}
\end{lemma}

\begin{proof}

At the beginning, let us consider all three cases simultaneously.

On a plane with the constant curvature we introduce a polar coordinate system with the origin at the point $O$ (see Fig.~\ref{ch1pic2}). Then, according to Lemma~\ref{ch1lem1}, the curvature of the curve $\gamma $ satisfies the equation 
\begin{equation}
\label{ch1eq4}
 k = \mu_0 (t) \cos \varphi - \frac{d\varphi }{dt}\sin \varphi ,
\end{equation}
where $\mu_0(t)$ is the geodesic curvature of the circle with the radius $t$ on a plane with the constant Gaussian curvature. 

\begin{figure}[h]
\begin{center}
\definecolor{uququq}{rgb}{0.25,0.25,0.25}
\begin{tikzpicture}[line cap=round,line join=round,>=triangle 45,x=1.0cm,y=1.0cm,scale=1.8]
\clip(-3.3,-1.5) rectangle (2.9,1.3);
\draw [shift={(-0.63,0.45)},line width=0.4pt,fill=black,fill opacity=0.1] (0,0) -- (48.33:0.22) arc (48.33:106.45:0.22) -- cycle;
\draw [shift={(2.56,-3.17)}] plot[domain=2.39:2.82,variable=\t]({1*5.97*cos(\t r)+0*5.97*sin(\t r)},{0*5.97*cos(\t r)+1*5.97*sin(\t r)});
\draw [shift={(7.11,-3.23)}] plot[domain=2.39:2.82,variable=\t]({1*5.97*cos(\t r)+0*5.97*sin(\t r)},{0*5.97*cos(\t r)+1*5.97*sin(\t r)});
\draw [shift={(0.31,-9.52)}] plot[domain=1.34:1.77,variable=\t]({1*10.63*cos(\t r)+0*10.63*sin(\t r)},{0*10.63*cos(\t r)+1*10.63*sin(\t r)});
\draw [shift={(-0.97,-11.68)}] plot[domain=1.34:1.77,variable=\t]({1*10.63*cos(\t r)+0*10.63*sin(\t r)},{0*10.63*cos(\t r)+1*10.63*sin(\t r)});
\draw [rotate around={13.68:(-0.31,0.01)},line width=1.2pt] (-0.31,0.01) ellipse (1.48cm and 0.51cm);
\draw [->] (-0.63,0.45) -- (-0.79,0.99);
\draw (-2.7,-0.59) node[anchor=north west] {$M^2$};
\draw (0.46,-0.24) node[anchor=north west] {$\gamma$};
\draw (-0.9,0.26) node[anchor=north west] {$t (s)$};
\draw (-1.35,0.15)-- (-1.16,-0.15);
\draw (-1.16,-0.15)-- (-0.63,0.45);
\draw [line width=0.4pt,dotted] (-1.16,-0.15)-- (1.08,0.45);
\draw [->] (-0.63,0.45) -- (-0.25,0.88);
\draw (-0.6,0.52) node[anchor=north west] {$P=\gamma(s)$};
\fill [color=black] (-1.16,-0.15) circle (1.0pt);
\draw[color=black] (-1.23,-0.32) node {$O$};
\fill [color=black] (-0.63,0.45) circle (1.0pt);
\fill [color=black] (-1.35,0.15) circle (1.0pt);
\draw[color=black] (-1.44,0.34) node {$Q_0$};
\fill [color=black] (1.08,0.45) circle (1.0pt);
\draw[color=black] (1.26,0.58) node {$Q_1$};
\draw[color=black] (-1.38,-0.02) node {$h$};
\draw[color=black] (-0.5,0.85) node {$\varphi$};
\end{tikzpicture}
\caption{}
\label{ch1pic2}
\end{center}
\end{figure}

Let us proceed with building the comparison object. We will take a circle $S$ of the curvature $k_0$ on a plane of the constant Gaussian curvature. Consider a point $O_1$ inside the circle at the distance $h$ from its border and introduce the polar coordinate system with the origin at $O_1$. We will denote the angle between the outer normal to the circle and the geodesic connecting $O_1$ and a point on $S$ as $\beta$ (see Fig.~\ref{ch1pic4}).

\begin{figure}[h]
\begin{center}
\definecolor{uququq}{rgb}{0.25,0.25,0.25}
\begin{tikzpicture}[line cap=round,line join=round,>=triangle 45,x=1.0cm,y=1.0cm,scale=1.5]
\clip(-2.8,-1.7) rectangle (0.5,1.8);
\draw [shift={(-0.72,1.06)},line width=0.4pt,fill=black,fill opacity=0.1] (0,0) -- (52.33:0.27) arc (52.33:76.38:0.27) -- cycle;
\draw(-1.04,-0.25) circle (1.35cm);
\draw [->] (-0.72,1.06) -- (-0.43,1.44);
\draw [->] (-0.72,1.06) -- (-0.61,1.52);
\draw (-2.39,-0.25)-- (-1.73,-0.25);
\draw (-1.73,-0.25)-- (-0.72,1.06);
\draw (-2.87,-0.21)  node[anchor=north west] {$\overline{Q_0}$};
\draw (-0.02,-0.91) node[anchor=north west] {$S$};
\fill [color=black] (-2.39,-0.25) circle (1.0pt);
\fill [color=black] (-1.73,-0.25) circle (1.0pt);
\draw[color=black] (-1.61,-0.45) node {$O_1$};
\fill [color=black] (-0.72,1.06) circle (1.0pt);
\draw[color=black] (-0.65,0.8) node {$P_1$};
\draw[color=black] (-2.05,-0.09) node {$h$};
\draw[color=black] (-0.43,1.64) node {$\beta$};
\end{tikzpicture}
\caption{}
\label{ch1pic4}
\end{center}
\end{figure}

According to Lemma~\ref{ch1lem1}
\begin{equation}
\label{ch1eq5}
k_0 = \mu_0(t) \cos \beta - \frac{d\beta }{dt} \sin \beta.
\end{equation}

Subtracting the equation~(\ref{ch1eq5}) from~(\ref{ch1eq4}) and taking into account the inequality $k \geqslant k_0$ we obtain
\begin{equation}
\label{ch1eq6}
\mu_0 ( \cos \varphi - \cos \beta ) - \frac{d\varphi }{dt}\sin \varphi  + \frac{d\beta }{dt}\sin \beta  = k - k_0\geqslant 0.
\end{equation}

Let us introduce the function $f (t):=\cos \varphi (t) - \cos \beta (t)$. It follows from~(\ref{ch1eq6}) that it satisfies the inequality
\begin{equation}
\label{ch1eq7} 
\begin{split}
&f'+\mu_0 f \geqslant 0,
\\
&f(h)=0.
\end{split} 
\end{equation}

The last condition $f(h)=0$ is true since in both cases $h$ is the distance from a point to a curve, thus $\varphi (h)=\beta (h)=0$.

Let us consider the arc of the curve $\gamma$ from the point $Q_0$ such that $dist(O,Q_0)=dist(O,\gamma)=h$ to the point $Q_1$ on which the function $t(s)$ is increasing (precisely on such arcs the curve can be parameterized by the distance parameter $t$). Herewith, we will be proving our statement for this arc.

A) First, we will prove our lemma in two more simple cases~--~Euclidean and Lobachevsky. It is known that on the Euclidean plane $\mu_0(t)=\frac{1}{t}>0$ and on the Lobachevsky plane of the curvature $-k^2_1, k_1>0$, $\mu_0(t)=k_1\cth k_1t > 0$. Thus, in both cases for all $t$, $\mu_0(t)>0$.

Let us show that for all points on the chosen arc $f(t)\geqslant 0$. 

Indeed, assume the contrary. Then for some points $f(t)<0$. Thus, since $f(h)=0$ then, according to Lemma~\ref{ch1lem5}, there is a point $\gamma(t_0)$ on the considered arc such that at $t_0$
\begin{equation}
\label{ch1eq8}
\begin{split}
&f(t_0)<0,
\\ &f'(t_0)<0.
\end{split} 
\end{equation}

But since $\mu_0>0$, the inequalities (\ref{ch1eq8}) contradict (\ref{ch1eq7}). Therefore, on the chosen arc $f(t)\geqslant 0$. It implies that $\cos\varphi (t)\geqslant \cos \beta (t)$. The estimates for $\cos\beta (t)$ were given in Lemma~\ref{ch1lem4}. This proves the statement of Lemma~\ref{ch1prooflem1} in our cases on the chosen arc. 

Since our curve $\gamma$ is regular, it can be represented as a union of such arcs, that differ from each other only by the minimal distance $h_i$ from the point $O$ to a particular arc. Estimating $\varphi$ separately on segments of the curve where the function $t = t(s)$ is monotonous, we will obtain the estimate for the whole closed curve. Here we should notice that the right-hand members of the estimates from Lemma~\ref{ch1lem4} are monotonous with respect to $h$. It means that we indeed can estimate the angle on every arc and then pick up the minimal value $h=\min  h_i$ over all arcs. 

B) Now we will move to the most complicated 3rd case. Straightforward calculations, that were made in A),  cannot be used here since the curvature of a circle on a sphere can be positive, negative and zero.

If the curve $\gamma $ is not a circle and the point $O$ is not its center, then $h <\frac{\pi }{ 2k_1}$. Indeed, by the condition of the lemma $k \geqslant 0$. It means that the curve $\gamma $ lies in the closed hemisphere, which implies the restriction on $h$. Let us show that for all $t$ close enough to $h$, $f(t) \geqslant 0$ and does not equal to zero unless in the neighborhood of $Q_0$ the arc of the curve is an arc a circle of the curvature $k_0$. 

Indeed, if arbitrary near to $h$ there exists a value $t$ such that $f(t)<0$, then according to Lemma~\ref{ch1lem5} among these values there exists $t_0$ close enough to $h$ such that 
\begin{equation}
\label{ch1eq9}
\begin{split}
&f(t_0)<0,
\\ &f'(t_0)<0.
\end{split}
\end{equation} 

Since $h < \frac{\pi }{2k_1}$, $\mu_0(h)>0$ and the point $\gamma(t_0)$ is arbitrary near to the point $\gamma(h)$, we have that $\mu_0 (t_0)>0$. But then the inequalities~(\ref{ch1eq9}) contradict the inequality~(\ref{ch1eq7}).

Let us take the value $t_1$ close enough to $h$ such that $f(t_1)>0$. We have just shown that such exists. We consider the Cauchy problem for the following differential equation:
\begin{equation}
\label{ch1eq10}
\begin{split}
&g' + \mu_0 (t) g = 0,
\\&g(t_1) = f(t_1)>0
\end{split}
\end{equation}

 In the 3rd case $\mu_0(t)=k_1 \ctg k_1t$. Thus, the solution of~(\ref{ch1eq10}) is
$$
g(t) = \frac{f (t_1) \sin k_1t_1}{\sin k_1t}. 
$$
Note that for $0<t \leqslant \frac{\pi }{2k_1}$ the function $g(t)>0$. What is more, in the interval $\frac{\pi }{2k_1}\leqslant t < \frac{\pi }{k_1}$ this function is monotonically increasing, i.e. $g'(t)>0$.

Let us compare the solutions of the inequality~(\ref{ch1eq7}) and the equation~(\ref{ch1eq10}) with the same initial condition (see Fig.~\ref{ch1pic5}). For those values of $t$, at which $f - g < 0$,
\begin{equation}
\label{ch1eq11}
\left( f - g \right)' \geqslant -\mu_0 \left( f - g \right) > 0
\end{equation}
for $t_1\leqslant t \leqslant \frac{\pi }{ 2k_1}$ (in this interval $\mu_0(t)\geqslant 0$). Since $f(t_1) - g(t_1) = 0$ (according to~(\ref{ch1eq10})), then by Lemma~\ref{ch1lem5} among the values satisfying $f - g < 0$, there exists the value $t_2$, such that $f(t_2) - g(t_2) < 0$, $f '(t_2 ) - g'(t_2) < 0$. This contradicts the inequality~(\ref{ch1eq11}). Therefore, for $t_1 \leqslant t\leqslant \frac{\pi }{2k_1}$ we have $f \geqslant g > 0$.

\begin{figure}[h]
\begin{center}
\definecolor{uququq}{rgb}{0.25,0.25,0.25}
\begin{tikzpicture}[line cap=round,line join=round,>=triangle 45,x=1.0cm,y=1.0cm,scale=2]
\clip(-1.2,-0.6) rectangle (4.4,3.65);
\draw [->] (0,0) -- (0,3.61);
\draw [->] (0,0) -- (4,0);
\draw [shift={(1.64,2.11)},line width=1.2pt]  plot[domain=3.5:4.78,variable=\t]({1*1.43*cos(\t r)+0*1.43*sin(\t r)},{0*1.43*cos(\t r)+1*1.43*sin(\t r)});
\draw [dotted] (0,1.6)-- (0.3,1.6);
\draw [dotted] (0.3,0)-- (0.3,1.6);
\draw [dotted] (1.75,0)-- (1.75,0.68);
\draw (-0.38,0)-- (0,0);
\draw (0,-0.34)-- (0,0);
\draw (3.6,-0.05) node[anchor=north west] {${t=t(s)}$};
\draw (3.2,-0.05) node[anchor=north west] {$\frac{\pi}{k_1}$};
\draw (1.51,-0.05) node[anchor=north west] {$\frac{\pi}{2k_1}$};
\draw (0.22,-0.05) node[anchor=north west] {$t_1$};
\draw (-1.22,1.88) node[anchor=north west] {$g(t_1) = f(t_1)$};
\draw (1.77,0.45) node[anchor=north west] {$g(\frac{\pi}{2k_1})$};
\draw (1.47,1.2) node[anchor=north west] {$f(\frac{\pi}{2k_1})$};
\draw (0.18,3.54) node[anchor=north west] {$g(t)$};
\draw (2.24,2.21) node[anchor=north west] {$f(t)$};
\draw [shift={(1.61,2.44)},line width=1.2pt]  plot[domain=3.18:4.47,variable=\t]({-0.05*1.77*cos(\t r)+-1*1.1*sin(\t r)},{1*1.77*cos(\t r)+-0.05*1.1*sin(\t r)});
\draw [dash pattern=on 2pt off 2pt] (3.49,3.79)-- (3.49,-0.35);
\draw[smooth,samples=100,domain=0.1:1.7453292519943295] plot(\x,{1/(\x+0.66)^4+0.4});
\draw[smooth,samples=100,domain=1.7453292519943295:3.246658503988659] plot(\x,{1/(\x-4)^4+0.39});
\draw [color=uququq] (0,0)-- ++(-1.5pt,0 pt) -- ++(3.0pt,0 pt) ++(-1.5pt,-1.5pt) -- ++(0 pt,3.0pt);
\draw [color=uququq] (0.3,0)-- ++(-1.5pt,0 pt) -- ++(3.0pt,0 pt) ++(-1.5pt,-1.5pt) -- ++(0 pt,3.0pt);
\fill [color=uququq] (0.3,1.6) circle (1.0pt);
\draw [color=uququq] (0,1.6)-- ++(-1.5pt,0 pt) -- ++(3.0pt,0 pt) ++(-1.5pt,-1.5pt) -- ++(0 pt,3.0pt);
\draw [color=uququq] (3.49,0)-- ++(-1.5pt,0 pt) -- ++(3.0pt,0 pt) ++(-1.5pt,-1.5pt) -- ++(0 pt,3.0pt);
\draw [color=uququq] (1.75,0)-- ++(-1.5pt,0 pt) -- ++(3.0pt,0 pt) ++(-1.5pt,-1.5pt) -- ++(0 pt,3.0pt);
\fill [color=uququq] (1.75,0.68) circle (1.0pt);
\fill [color=uququq] (1.75,0.43) circle (1.0pt);
\end{tikzpicture}
\caption{}
\label{ch1pic5}
\end{center}
\end{figure}

For $t > \frac{\pi }{ 2k_1}$, $\mu_0(t)<0$, $f(\frac{\pi }{2k_1}) - g(\frac{\pi }{ 2k_1}) \geqslant 0$. Thus, from the inequality~(\ref{ch1eq11}) it follows that $f' - g' \geqslant 0$, i.e. $f' \geqslant g' > 0$. It means that for $t\geqslant \frac{\pi }{2k_1}$ on a segment of the curve $\gamma $ where $t=t(s)$ is monotonically increasing function, $f$ is monotonically increasing too. Since $f(\frac{\pi }{2k_1})\geqslant 0$, then in the biggest interval from the value $\frac{\pi }{2k_1}$ where the functions are defined, we get  $f\geqslant 0$. 

Summing up, we have shown that $\cos \varphi (t) \geqslant \cos\beta(t)$ on the chosen arc. Using the estimate for $\cos\beta(t)$ from Lemma~\ref{ch1lem4}, we obtain the statement of Lemma~\ref{ch1prooflem1} on the chosen arc. Then, applying exactly the same idea from A) about the partition of $\gamma$ into arcs for which the distance is monotonous we will get the angle estimate for the closed curve. Finally, even if $h_i > \frac{\pi }{2k_1}$ for some arcs, then $f(h_i) = 0$ and from the inequality~(\ref{ch1eq7}) it follows that $f'>0$, $f>0$ on them. 

Thereby, the 3rd case alongside with Lemma~\ref{ch1prooflem1} are proved in full generality. 

\end{proof}

Now we are ready to prove Theorems~\ref{ch1th1} and~\ref{ch1th2}.

Let us introduce a polar coordinate system with the origin at the point $O$. Then the arc length will be $ds^2 = dt^2 + g_{ij} d\theta^i d\theta^j$. Since the boundedness of the normal curvature, the hypersurface $\partial\Omega$ will be convex, embedded and compact. Moreover, in all the cases the hypersurface will lie in the domain of regularity of such coordinate system. Then we can assume that the manifold $\partial\Omega$ is the $0$-level set of the function $F(t,\theta) = t - \rho(\theta)$.

Let $\gamma$ be the integral trajectory of the vector field $Y = \grad_{\partial\Omega} \rho/|\grad_{\partial\Omega} \rho|$. Denoting a point at the distance $h$ from the point $O$ as $Q_0 \in \gamma$ we obtain $Y(Q_0)=0$, $\varphi (Q_0) = 0$. Let $P \in \gamma$ be a point at the distance $h_1$ from $O$ such that on the arc $Q_0 P\subset\gamma$ the distance function from the point $O$ to points on $\partial\Omega$ is monotonous. Then $\gamma$ can be parameterized by the distance parameter $t \in (h; h_1]$ from $O$ to points on $\partial\Omega$.

By Lemma~\ref{ch1lem1}, taking into account the remark on it,  at points on $\gamma$
\begin{equation}
\label{ch1eq12}
k_n (t) = \mu_n (t) \cos \varphi (t)  - \frac{d\varphi}{dt}\sin \varphi (t).
\end{equation}

Similarly to the two-dimensional case, let us consider a circle $S$ with the curvature $k_0$ on a two-dimensional plane with the constant Gaussian curvature (equal to $0$ in the Euclidean case, $k_1^2$ in the spherical case, $-k_1^2$ in the Lobachevsky case). Let $\overline{Q_0}$ be a point on this circle, $O_1$~---~a point at the distance $h$ from $\overline{Q_0}$ lying on the geodesic line perpendicular to $S$ at the point $\overline{Q_0}$. Let $\beta$ be the angle between the outward normal vector and the geodesic line from $O_1$ taken at a point on $S$ (see Fig.~\ref{ch1pic4}). Then, by the same Lemma~\ref{ch1lem1},
\begin{equation}
\label{ch1eq13}
k_0 =  \mu_0 (t) \cos \beta(t) -  \frac{d\beta}{dt} \sin \beta(t).
\end{equation}

Let us subtract the equation~(\ref{ch1eq13}) from~(\ref{ch1eq12})  
\begin{equation}
\label{ch1eq14}
\mu_n (t) \cos\varphi - \mu_0(t) \cos\beta -  \frac{d\varphi }{dt}\sin\varphi + \frac{d\beta }{dt}\sin\beta  = k - k_0 \geqslant 0.
\end{equation}

According to Lemma~\ref{ch1lem3} we have $\mu_n (t)\leqslant \mu_0(t)$. Then~(\ref{ch1eq14}) can be rewritten as
$$
\mu_0( \cos\varphi - \cos\beta ) + \frac{d}{dt}\left(\cos\varphi  - \cos\beta \right)\geqslant k - k_0 \geqslant 0.
$$
And this inequality coincides with the inequality~(\ref{ch1eq7}). 

Thus, the computations from the proof of Lemma~\ref{ch1prooflem1} become valid and, thereby, prove Theorems~\ref{ch1th1} and~\ref{ch1th2}.


\section{Proofs of the estimates for the width of a spherical layer}


\subsection{Auxiliary results necessary for the proof of Theorem~\ref{ch2th1}}

\begin{lemma}
\label{ch2lem1}

Let $\partial\Omega$ be a complete $k_0$-convex hypersurface in a Riemannian manifold $M^{n+1}(c)$ of the constant sectional curvature $c$, where 
\begin{enumerate}
\item for $c=0$ or $c = k_1^2$ ($k_1>0$), $k_0>0$;
\item $c=-k_1^2$ ($k_1>0$), $k_0>k_1$.
\end{enumerate}
Then $\partial\Omega$ is an embedded convex hypersurface such that at any point $P\in \partial\Omega$ there is a locally supporting for $\partial\Omega$ sphere of the radius $R$ which encloses the whole hypersurface.
\end{lemma}

\begin{remark}
As we have remarked above, here by $R$ we understand the radius of a circle of the curvature $k_0$ lying in a two-dimensional manifold $M^2(c)$ of the constant Gaussian curvature $c$.
\end{remark}

\begin{proof}

For $C^k$-smooth hypersurfaces, $k \geqslant 2$, the assertion is directly follows from~\cite{Kar}.

Let now $\partial\Omega$ be a non-smooth $k_0$-convex hypersurface. For sufficiently small $\tau$ let us consider external equidistant surfaces $\partial\Omega_\tau$ which will be $\varepsilon(\tau)$-convex with $\varepsilon(\tau) \to k_0-0$ when $\tau \to 0$ and $\varepsilon(\tau) > 0$ (for the 1st case) or $\varepsilon(\tau) > k_1$ (in the 2nd case). It is known that $\partial\Omega_\tau$ is a $C^{1,1}$-smooth hypersurface. Thus, it can be approximated with regular hypersurfaces $\partial\Omega_{\tau,\delta}$, whose normal curvatures $k_n \geqslant \varepsilon(\tau) - \nu(\delta)$ with $\nu(\delta) \to 0+0$ when $\delta \to 0$ and $\varepsilon(\tau) - \nu(\delta) > 0$ (in the 1st case) or $\varepsilon(\tau) - \nu(\delta) > k_1$ (in the 2nd).

By the already proved regular case, $\partial\Omega_{\tau,\delta}$ is enclosed by the sphere with the radius $R_{\tau,\delta}$ of the curvature $\varepsilon(\tau) - \nu(\delta)$ supporting for the surface at an arbitrary point $P_{\tau,\delta} \in \partial \Omega_{\tau,\delta}$. Taking limits as $\tau \to 0$, $\delta \to 0$ we will obtain that the sphere with the radius $R = \lim\limits_{\tau,\delta \to 0} R_{\tau,\delta}$ supporting for $\partial\Omega=\lim \limits_{\tau,\delta \to 0}\partial\Omega_{\tau,\delta}$ at the point $P = \lim \limits_{\tau,\delta \to 0}P_{\tau,\delta}$, $P \in \partial\Omega$ will enclose the hypersurface $\partial\Omega$. This holds for an arbitrary point $P$. Thus, the lemma in the non-regular case is proved. 

\end{proof}

\begin{figure}[h]
\begin{center}
\definecolor{uququq}{rgb}{0.25,0.25,0.25}
\begin{tikzpicture}[line cap=round,line join=round,>=triangle 45,x=1.0cm,y=1.0cm,scale=1.5]
\clip(-3.3,-1.5) rectangle (2.9,1.3);
\draw [shift={(2.56,-3.17)}] plot[domain=2.39:2.82,variable=\t]({1*5.97*cos(\t r)+0*5.97*sin(\t r)},{0*5.97*cos(\t r)+1*5.97*sin(\t r)});
\draw [shift={(7.11,-3.23)}] plot[domain=2.39:2.82,variable=\t]({1*5.97*cos(\t r)+0*5.97*sin(\t r)},{0*5.97*cos(\t r)+1*5.97*sin(\t r)});
\draw [shift={(0.31,-9.52)}] plot[domain=1.34:1.77,variable=\t]({1*10.63*cos(\t r)+0*10.63*sin(\t r)},{0*10.63*cos(\t r)+1*10.63*sin(\t r)});
\draw [shift={(-0.97,-11.68)}] plot[domain=1.34:1.77,variable=\t]({1*10.63*cos(\t r)+0*10.63*sin(\t r)},{0*10.63*cos(\t r)+1*10.63*sin(\t r)});
\draw [shift={(-0.58,-0.34)}] plot[domain=-4.15:0.45,variable=\t]({1*1.12*cos(\t r)+-0.03*0.58*sin(\t r)},{0.03*1.12*cos(\t r)+1*0.58*sin(\t r)});
\draw [shift={(-0.58,-0.34)},line width=1.6pt]  plot[domain=0.45:2.14,variable=\t]({1*1.12*cos(\t r)+-0.03*0.58*sin(\t r)},{0.03*1.12*cos(\t r)+1*0.58*sin(\t r)});
\draw [shift={(-0.19,0.41)}] plot[domain=-4.15:0.45,variable=\t]({-1*1.12*cos(\t r)+0.03*0.58*sin(\t r)},{-0.03*1.12*cos(\t r)+-1*0.58*sin(\t r)});
\draw [shift={(-0.19,0.41)},line width=1.6pt]  plot[domain=0.45:2.14,variable=\t]({-1*1.12*cos(\t r)+0.03*0.58*sin(\t r)},{-0.03*1.12*cos(\t r)+-1*0.58*sin(\t r)});
\draw (-2.69,-0.54) node[anchor=north west] {$M^2(c)$};
\draw [->,line width=0.4pt] (-0.58,-0.34) -- (-1.51,-0.67);
\fill [color=black] (-1.19,0.13) circle (1.0pt);
\draw[color=black] (-1.43,0.23) node {$A$};
\fill [color=black] (0.43,-0.06) circle (1.0pt);
\draw[color=black] (0.69,-0.09) node {$B$};
\fill [color=black] (-0.58,-0.34) circle (1.0pt);
\fill [color=black] (-0.19,0.41) circle (1.0pt);
\draw[color=black] (-1.33,-0.41) node {$R$};
\end{tikzpicture}
\caption{}
\label{ch2pic1}
\end{center}
\end{figure}

Further we will need the following observation. Let $A$ and $B$ be two arbitrary points in $\Omega \subset M^{n+1} (c)$, where $\Omega$ is the domain, bounded by $\partial\Omega$. Consider an arbitrary totally geodesic two-dimensional submanifold $M^2 (c)$ in $M^{n+1} (c)$ passing through $A$ and $B$ (see Fig.~\ref{ch2pic1}). In general situation, in $M^2(c)$ there are precisely two circles of the radius $R$ passing through $A$ and $B$. Both of these circles are divided by these points into two arcs~--~smaller and bigger. Hereafter, we will call a smaller arc of the circle of the radius $R$ passing trough the points $A$ and $B$ as \textit{a smaller circular arc of the radius $R$ for points $A$ and $B$}.        

The following lemma holds.

\begin{lemma}

\label{ch2lem2}

Let $\partial\Omega$ be a complete $k_0$-convex hypersurface in a complete simply connected Riemannian manifold $M^{n+1} (c)$ of the constant sectional curvature $c$ (for $c=0$ or $c=k_1^2$, $k_0>0$; for  $c = -k_1^2$, $k_0 > k_1>0$); $\Omega$ is the domain enclosed by the hypersurface. Then any smaller circular arc of the radius $R$ passing through any two points $A,B\in\Omega$ lies in $\Omega$.

\end{lemma}

\begin{proof}

Let us assume the contrary, implying that there are $A,B\in\Omega$ and some smaller circular arc $\omega$ for the points $A$ and $B$ which does not lie entirely in the domain $\Omega$. 

Consider the cross-section of $\Omega$ by the two-dimensional subspace $M^2 (c)$ that contains $A,B$, and $\omega$. Denote as $\gamma := M^2 (c) \cap \partial\Omega$ the curve in this section (see. Fig.~\ref{ch2pic2}).           

\begin{figure}[h]
\begin{center}
\definecolor{uququq}{rgb}{0.25,0.25,0.25}
\begin{tikzpicture}[line cap=round,line join=round,>=triangle 45,x=1.0cm,y=1.0cm,scale=2.5]
\clip(-2.4,-1.5) rectangle (1.1,1.4);
\draw [rotate around={41.24:(-0.63,-0.23)}] (-0.63,-0.23) ellipse (1.4cm and 0.98cm);
\draw [shift={(-0.34,0.29)}] plot[domain=-1.28:1.52,variable=\t]({0.1*0.94*cos(\t r)+-1*0.77*sin(\t r)},{1*0.94*cos(\t r)+0.1*0.77*sin(\t r)});
\draw (-1.03,0.69)-- (0.35,0.77);
\draw [shift={(0.06,-1.69)},line width=1.2pt]  plot[domain=1.51:1.98,variable=\t]({1*2.66*cos(\t r)+0*2.66*sin(\t r)},{0*2.66*cos(\t r)+1*2.66*sin(\t r)});
\draw [shift={(-0.63,-0.23)},line width=1.2pt]  plot[domain=0.11:1.35,variable=\t]({0.75*1.4*cos(\t r)+-0.66*0.98*sin(\t r)},{0.66*1.4*cos(\t r)+0.75*0.98*sin(\t r)});
\draw [shift={(-0.34,0.29)},line width=1.2pt]  plot[domain=-0.95:1.21,variable=\t]({0.1*0.94*cos(\t r)+-1*0.77*sin(\t r)},{1*0.94*cos(\t r)+0.1*0.77*sin(\t r)});
\draw [shift={(0.06,-1.69)}] plot[domain=1.98:2.62,variable=\t]({1*2.66*cos(\t r)+0*2.66*sin(\t r)},{0*2.66*cos(\t r)+1*2.66*sin(\t r)});
\draw [shift={(0.06,-1.69)}] plot[domain=1.19:1.51,variable=\t]({1*2.66*cos(\t r)+0*2.66*sin(\t r)},{0*2.66*cos(\t r)+1*2.66*sin(\t r)});
\draw (-2.15,0.25) node[anchor=north west] {$\delta$};
\draw (0.19,0.72) node[anchor=north west] {$\omega$};
\draw (-0.03,0.92) node[anchor=north west] {$\gamma_1$};
\draw (-0.09,1.18) node[anchor=north west] {$\delta_1$};
\draw[color=black] (-1.6,-0.24) node {$\gamma$};
\fill [color=uququq] (-1.1,0.41) circle (0.5pt);
\draw[color=uququq] (-1.2,0.33) node {$A$};
\fill [color=uququq] (0.43,0.48) circle (0.5pt);
\draw[color=uququq] (0.33,0.41) node {$B$};
\draw[color=black] (-0.09,1.33) node {$\omega_1$};
\fill [color=uququq] (-1.03,0.69) circle (0.5pt);
\draw[color=uququq] (-0.95,0.58) node {$A_1$};
\fill [color=uququq] (0.35,0.77) circle (0.5pt);
\draw[color=uququq] (0.46,0.82) node {$B_1$};
\fill [color=uququq] (-0.44,0.92) circle (0.5pt);
\draw[color=uququq] (-0.47,1.04) node {$C_1$};
\fill [color=uququq] (-0.99,0.75) circle (0.5pt);
\draw[color=uququq] (-1.07,0.9) node {$X_1$};
\fill [color=uququq] (0.22,0.97) circle (0.5pt);
\draw[color=uququq] (0.3,1.06) node {$Y_1$};
\end{tikzpicture}
\caption{}
\label{ch2pic2}
\end{center}
\end{figure}

It is known that if intersect a $\lambda$-convex hypersurface with any two-dimentional totally-geodesic subspace, then we obtain the $\lambda$-convex curve. Thus, $\gamma$ is $k_0$-convex.

Let $A_1, B_1$ be the intersection points of $\omega$ and $\gamma$. If to denote the part of  $\omega$ that lies between $A_1$ and $B_1$ as $\omega_1$, and the part of $\gamma$ bounded by $\omega$ and the chord $AB$ as $\gamma_1$, then $\omega_1$ and $\gamma_1$ will be convex curves lying on the same side from the geodesic that connects $A_1$ and $B_1$. 

Now, let $C_1$ be an arbitrary point on $\gamma_1$ distinct from $A_1$ and $B_1$. Since $\gamma$ is a $k_0$-convex closed curve, by Lemma~\ref{ch2lem1} the circle $\delta$ of the radius $R$ locally supporting for $\gamma$ at the point $C_1$ encloses $\gamma$. Suppose that $\delta$ intersects $\omega_1$ in two points $X_1$ and $Y_1$. If $X_1$ or $Y_1$ coincides with $A_1$ or $B_1$, then since the arbitrariness of the choice of $C_1$ we get $\omega_1 \equiv \gamma_1$, which contradicts the fact that $\omega_1$ does not lie in $\Omega$. Thus, $X_1 \ne A_1$, $Y_1 \ne B_1$.

But then, since $\omega_1$ is a smaller circular arc of the radius $R$, the arc $\delta_1$ of the circle $\delta$, bounded by $\omega_1$ and $A_1B_1$, is less than a half of the circle $\delta$. And taking into account the convexity argument, $\delta_1$ and $\omega_1$ lie on the same side with respect to the geodesic $X_1 Y_1$. 

Thereby, we obtain that for two given points $X_1$ and $Y_1$ there exist two distinct smaller circular arcs of the radius $R$ that lie on the same side from the geodesic $X_1 Y_1$, which is impossible. From this the assertion of Lemma~\ref{ch2lem2} follows.          

\end{proof}

\subsection{Estimates for the class of spindle-shaped hypersurfaces}

In this section we will build the key object for estimating the width of a spherical layer.

Similarly to the above, let us consider a complete simply connected Riemannian manifold $M^{n+1}(c)$ of the constant sectional curvature $c$ and let us denote by $R$ the radius of the circle of the curvature $k_0$ in a two-dimensional manifold $M^2(c)$.
 
For fixed points $P$, $Q\in M^{n+1}(c)$ let us consider the special class of \textit{spindle-shaped hypersurfaces} $v(P,Q)$ which are obtained by rotating the smaller circular arc of the radius $R$ for the points $P$ and $Q$ around the geodesic $l$ passing through these points (see Fig.~\ref{ch2pic3}).

\begin{figure}[h]
\begin{center}
\definecolor{uququq}{rgb}{0.25,0.25,0.25}
\begin{tikzpicture}[line cap=round,line join=round,>=triangle 45,x=1.0cm,y=1.0cm,scale=2]
\clip(-1.7,-1.95) rectangle (2.5,1.55);
\draw [shift={(0.38,-1.49)},line width=1.2pt]  plot[domain=0.69:2.45,variable=\t]({1*2.05*cos(\t r)+0*2.05*sin(\t r)},{0*2.05*cos(\t r)+1*2.05*sin(\t r)});
\draw [shift={(0.38,1.11)},line width=1.2pt]  plot[domain=3.83:5.6,variable=\t]({1*2.05*cos(\t r)+0*2.05*sin(\t r)},{0*2.05*cos(\t r)+1*2.05*sin(\t r)});
\draw [shift={(0.38,-0.19)}] plot[domain=-3.14:0,variable=\t]({0*0.75*cos(\t r)+1*0.23*sin(\t r)},{-1*0.75*cos(\t r)+0*0.23*sin(\t r)});
\draw [shift={(0.38,-0.19)},dash pattern=on 1pt off 1pt]  plot[domain=0:pi,variable=\t]({0*0.75*cos(\t r)+1*0.23*sin(\t r)},{-1*0.75*cos(\t r)+0*0.23*sin(\t r)});
\draw [line width=0.4pt,dotted,fill=black,fill opacity=0.1] (0.38,-0.19) circle (0.75cm);
\draw [shift={(0.38,-0.19)},line width=0.4pt]  plot[domain=-3.14:0,variable=\t]({1*0.75*cos(\t r)+0*0.29*sin(\t r)},{0*0.75*cos(\t r)+1*0.29*sin(\t r)});
\draw [shift={(0.38,-0.19)},line width=0.4pt,dash pattern=on 1pt off 1pt]  plot[domain=0:pi,variable=\t]({1*0.75*cos(\t r)+0*0.29*sin(\t r)},{0*0.75*cos(\t r)+1*0.29*sin(\t r)});
\draw [shift={(0.38,4.07)},line width=1.2pt]  plot[domain=4.36:5.07,variable=\t]({1*4.55*cos(\t r)+0*4.55*sin(\t r)},{0*4.55*cos(\t r)+1*4.55*sin(\t r)});
\draw [shift={(0.38,-4.46)},dash pattern=on 1pt off 1pt]  plot[domain=1.22:1.93,variable=\t]({1*4.55*cos(\t r)+0*4.55*sin(\t r)},{0*4.55*cos(\t r)+1*4.55*sin(\t r)});
\draw [line width=0.4pt,dotted,fill=black,fill opacity=0.05] (0.38,-0.19) circle (1.58cm);
\draw [->] (0.38,-0.19) -- (0.95,0.29);
\draw [->] (0.95,0.29) -- (1.59,0.83);
\draw (-0.94,0.74) node[anchor=north west] {$v(P,Q)$};
\draw [line width=0.4pt] (-1.6,-0.19)-- (2.39,-0.19);
\draw (2.25,0.13) node[anchor=north west] {${l}$};
\fill [color=uququq] (-1.2,-0.19) circle (0.5pt);
\draw[color=uququq] (-1.29,-0.35) node {$Q$};
\fill [color=uququq] (1.96,-0.19) circle (0.5pt);
\draw[color=uququq] (2.05,-0.35) node {$P$};
\fill [color=uququq] (0.38,-0.19) circle (0.5pt);
\draw[color=uququq] (0.32,-0.35) node {$O$};
\draw[color=black] (-0.08,0.3) node {$S$};
\draw[color=black] (-0.31,1.1) node {$S_1$};
\draw[color=black] (0.72,0.27) node {$r$};
\draw[color=black] (1.37,0.81) node {$\rho$};
\end{tikzpicture}
\caption{}
\label{ch2pic3}
\end{center}
\end{figure}

Note that $v(P,Q)$ is a $k_0$-convex surface of revolution. For any two-dimensional plane $M^2(c)$ which contains the points $P$ and $Q$ we have that $M^2(c)\cap v(P,Q)$ is the curve $\gamma$ composed of the two symmetric with respect to $PQ$ smaller circular arcs of the radius $R$. Hereby, we will call such curves as \textit{lunes}.     

Let $O\in l$ be a point equidistant from $P$ and $Q$. Since $v(P,Q)$ is obtained by rotating a circular arc, the point $O$ is the center of the ball $S$ inscribe into $v(P,Q)$. Let $r$ be the radius of $S$. Then $\omega := M^2(c)\cap S$ will be a circle with the center $O$ and radius $r$ inscribed in $\gamma$.  
  
Due to the fact that $v(P,Q)$ was build by rotating the smaller circular arc, the sphere $S_1$ with the center $O$ and of the radius $\rho:=|OP|=|OQ|$ will be the circumscribe sphere for $v(P,Q)$. Here $|OP|$ and $|OQ|$ are the lengths of the corresponding geodesic segments.

It is obvious that with the given radius of the inscribe sphere and fixed $R$ we can uniquely reconstruct the points $P$ and $Q$, and thus the whole hypersurface $v(P,Q)$ (since with the given radius $R$ of a circle and the height $r$ of a circular segment one can uniquely rebuild the corresponding circular arc and its endpoints).

Thereby, we can consider the class of spindle-shaped hypersurfaces parameterized by the value of $r$. Note that $r\in[0, R]$.

Then, by construction and since $\rho$ is uniquely defined by $r$ either, every spindle-shaped hypersurface can be enclosed in the spherical layer of the width $d=d(r)=\rho(r)-r$.

It appears that the following lemma holds.

\begin{lemma}
\label{ch2lem3}
For the class of spindle-shaped hypersurfaces in the manifolds of the constant sectional curvature the following estimates for the width $d=d(r)$ of a spherical layer holds:
\begin{enumerate}
\item 
For the Euclidean space
$$
d=d(r)\leqslant d_0 = \frac{\sqrt{2} - 1}{k_0}
$$
and the equality is attained with $$r=r_0=\frac{1}{k_0(2+\sqrt{2})}.$$

\item 
For the spherical space $S^{n+1}(k_1^2)$ of the curvature $k_1^2$
$$
d=d(r) \leqslant d_0 = \frac{2}{k_1}\arccos\sqrt{\cos k_1R} - R
$$
and the equality is attained with $$r=r_0=R-\frac{1}{k_1}\arccos\sqrt{\cos k_1R}.$$

\item For the Lobachevsky space $\mathbb H^{n+1}(-k_1^2)$ of the curvature $-k_1^2$
$$
d=d(r)\leqslant d_0 = \frac{2}{k_1}\arcch\sqrt{\ch k_1R} - R
$$
and the equality is attained with $$r=r_0=R-\frac{1}{k_1}\arcch\sqrt{\ch k_1R}.$$
\end{enumerate}
\end{lemma}

\begin{proof}

In order to prove the assertion we will calculate directly the value $d(r)=\rho(r)-r$ and its extremum.

Let us consider an arbitrary two-dimensional totally geodesic submanifold $M^2(c)$ passing through the points $P$ and $Q$. As above $\gamma := M^2(c) \cap v(P,Q)$.

Let $O_1$ be the center of one of the circles of the radius $R$ which contains one of the smaller circular arc for the points $P$ and $Q$; $A\in\gamma$ is the intersection point of the geodesic $O_1O$ with $\gamma$ that does not lie between $O$ and $O_1$ (see Fig.~\ref{ch2pic4}).

\begin{figure}[h]
\begin{center}
\definecolor{uququq}{rgb}{0.25,0.25,0.25}
\begin{tikzpicture}[line cap=round,line join=round,>=triangle 45,x=1.0cm,y=1.0cm,scale=2]
\clip(-2.6,-2) rectangle (1.3,1.3);
\draw (-2.16,-0.35)-- (0.83,-0.35);
\draw [shift={(-0.67,-1.59)},dash pattern=on 1pt off 1pt]  plot[domain=2.45:2.89,variable=\t]({1*1.94*cos(\t r)+0*1.94*sin(\t r)},{0*1.94*cos(\t r)+1*1.94*sin(\t r)});
\draw [shift={(-0.67,-1.59)},dash pattern=on 1pt off 1pt]  plot[domain=0.25:0.69,variable=\t]({1*1.94*cos(\t r)+0*1.94*sin(\t r)},{0*1.94*cos(\t r)+1*1.94*sin(\t r)});
\draw [shift={(-0.67,-1.59)},line width=1.2pt]  plot[domain=0.69:2.45,variable=\t]({1*1.94*cos(\t r)+0*1.94*sin(\t r)},{0*1.94*cos(\t r)+1*1.94*sin(\t r)});
\draw [shift={(-0.67,0.88)},line width=1.2pt]  plot[domain=3.83:5.59,variable=\t]({1*1.94*cos(\t r)+0*1.94*sin(\t r)},{0*1.94*cos(\t r)+1*1.94*sin(\t r)});
\draw(-0.67,-0.35) circle (1.5cm);
\draw(-0.67,-0.35) circle (0.7cm);
\draw (-0.67,0.35)-- (-0.67,-1.59);
\draw (-0.67,-1.59)-- (0.83,-0.35);
\draw (-0.68,0.12) node[anchor=north west] {$r$};
\draw (-0.7,-0.55) node[anchor=north west] {$R-r$};
\draw (-0.05,-1.05) node[anchor=north west] {$R$};
\draw (-1.79,-0.78) node[anchor=north west] {$\gamma$};
\draw (0.14,-0.09) node[anchor=north west] {$\rho$};
\fill [color=uququq] (-2.16,-0.35) circle (0.5pt);
\draw[color=uququq] (-2.29,-0.24) node {$Q$};
\fill [color=uququq] (0.83,-0.35) circle (0.5pt);
\draw[color=uququq] (0.97,-0.24) node {$P$};
\fill [color=uququq] (-0.67,-0.35) circle (0.5pt);
\draw[color=uququq] (-0.79,-0.48) node {$O$};
\fill [color=uququq] (-0.67,-1.59) circle (0.5pt);
\draw[color=uququq] (-0.75,-1.7) node {$O_1$};
\fill [color=uququq] (-0.67,0.35) circle (0.5pt);
\draw[color=uququq] (-0.76,0.52) node {$A$};
\draw[color=black] (-1.26,0.82) node {$\omega_1$};
\draw[color=black] (-1.05,0.06) node {$\omega$};
\end{tikzpicture}
\caption{}
\label{ch2pic4}
\end{center}
\end{figure}

Since $O$ is the midpoint of $PQ$, the geodesic $O_1O$ is perpendicular to the geodesic $PQ$ and the point $A$ is the tangency point of $\omega$ and $\gamma$.

By construction, $AO=r$, $OQ=\rho$, $OO_1 = R-r$, $O_1Q=R$, $\angle QOO_1 = \frac{\pi}{2}$. As we have already noticed  
\begin{equation}
\label{ch2eq4}
0 \leqslant r \leqslant R.
\end{equation} 

Let us obtain the estimates for the possible ambient spaces.

1) For $c=0$, $M^2(c)=\mathbb E^2$, $R=\frac{1}{k_0}$. From the right triangle $\triangle O_1OQ$:
 \begin{equation}
\label{ch2eq5}
\rho=\rho(r)=\sqrt{\frac{2r}{k_0}-r^2}.
\end{equation} 

Thus,
\begin{equation}
\label{ch2eq6}
d(r)=\rho(r) - r = \sqrt{\frac{2r}{k_0}-r^2} - r.
\end{equation}

Taking into account~(\ref{ch2eq4}), let us find the maximum of the function $d(r)$ in the interval $\left[0,\frac{1}{k_0}\right]$: $d(0)=0$, $d(\frac{1}{k_0})=0$, $d\geqslant 0$. Which means that the maximum is attained in $(0, \frac{1}{k_0})$. It is easy to check that $r_0=\frac{1}{k_0(2+\sqrt{2})}<\frac{1}{k_0}$ is the maximum point and 
\begin{equation*}
\begin{split}
d_0&=\max\limits_{r\in\left[0,\frac{1}{k_0}\right]} d(r)=d(r_0)=\frac{\sqrt{2}-1}{k_0}.
\end{split}
\end{equation*}

Therefore, for all $r$ we have $d(r)=\rho(r)-r \leqslant d_0 = \frac{\sqrt{2} - 1}{k_0}$, which proves the estimate~(\ref{ch2eq1}) and the equality case. 

2) For $c=k_1^2$, $k_1>0$, $M^2(c)=S^2(k_1^2)$, $R=\frac{1}{k_1}\arccos \frac{k_0}{k_1}$. From the right triangle $\triangle O_1OQ$ on the sphere: $\cos k_1 |O_1Q| = \cos k_1| OQ| \cos k_1 |OO_1|$, where $|\cdot|$ means the length of the corresponding geodesic segment on the sphere. Thus, in our notations
\begin{equation}
\label{ch2eq9}
\rho=\rho(r)=\frac{1}{k_1}\arccos \frac{\cos k_1R}{\cos k_1(R-r)}.
\end{equation}
   
Thereby, 
\begin{equation}
\label{ch2eq10}
d(r)=\rho(r)-r=\frac{1}{k_1}\arccos \frac{\cos k_1R}{\cos k_1(R-r)}-r.
\end{equation}

Let us find the maximum of this function in the interval $[0,R]$. $d(0)=0$, $d(R)=\frac{1}{k_0}k_0R - R =0$, $d\geqslant 0$. Thus, similarly to 1), the maximum point lies in $(0,R)$.

The function $d(r)$ attains its maximum simultaneously with the function $\tg k_1 d(r)$. Let us compute the tangent from the both sides in ~(\ref{ch2eq10}). For doing this, we introduce the following notations:
\begin{equation}
\label{ch2eq11}
f:=\sqrt{\cos^2 k_1(R-r)-\cos^2k_1R},
\end{equation}
\begin{equation}
\label{ch2eq12}
\alpha:=\arccos \frac{\cos k_1R}{\cos k_1(R-r)}.
\end{equation}
From~(\ref{ch2eq12})
\begin{equation}
\label{ch2eq14}
\tg \alpha = \frac{f}{\cos k_1 R},
\end{equation}
\begin{equation}
\label{ch2eq15}
\beta := k_1r, \, \, \tg \beta = \tg k_1r.
\end{equation}
Using~(\ref{ch2eq10}), (\ref{ch2eq11}), (\ref{ch2eq14}) and (\ref{ch2eq15}), 
\begin{equation}
\label{ch2eq16}
\tg k_1 d(r) = \frac{\frac{f}{\cos k_1R} -\tg k_1r}{1+\frac{f}{\cos k_1R}\tg k_1r} = \frac{f - \cos k_1 R \tg k_1 r}{f \tg k_1 r + \cos k_1R}.
\end{equation}
Taking the derivative of~(\ref{ch2eq16}) with respect to $r$ and equating it to zero we come to:

\begin{equation}
\label{ch2eq18}
f' \cos k_1R = k_1 \cos^2 k_1(R-r).
\end{equation}

From ~(\ref{ch2eq11}), $f'=\frac{k_1\cos k_1(R-r)\sin k_1(R-r)}{\sqrt{\cos^2 k_1 (R-r) - \cos^2 k_1R}}$. Substituting it to~(\ref{ch2eq18}) and making all necessary cancellations we will get  
\begin{equation}
\label{ch2eq19}
\cos k_1 R = \cos^2 k_1 (R-r).
\end{equation}

Thereby, from~(\ref{ch2eq19}) the maximum point (bearing in mind that $\cos k_1R > 0$) is
\begin{equation*}
r_0 = R - \frac{1}{k_1}\arccos\sqrt{\cos k_1R}.
\end{equation*} 

Finally, from~(\ref{ch2eq10}) and~(\ref{ch2eq19}) we obtain
$$
d(r)\leqslant d_0 = \max\limits_{r\in\left[0, R\right]} d(r)=d(r_0)=\frac{2}{k_1}\arccos\sqrt{\cos k_1R}-R,
$$
as required. The case 2) is proved.

3) The proof for the case 3) is identical to the case 2) with the change of the spherical trigonometry to the hyperbolic one. 

\end{proof}

\begin{remark}
With $k_1 \to 0$ the metrics of the spaces $S^{n+1}(k_1^2)$ and $\mathbb H^{n+1}(-k_1^2)$ tend to the euclidean metric. It is easy to show that the estimates~(\ref{ch2eq2}) and~(\ref{ch2eq3}) from Theorem~\ref{ch2th1} approach the estimate~(\ref{ch2eq1}) as $k_1 \to 0$. 
\end{remark}

\subsection{Proof of Theorem~\ref{ch2th1}}

We are going to prove Theorem~\ref{ch2th1} in all the cases simultaneously pointing out the differences when necessary.

Let us start with the \textit{regular case}. So we suppose that $\partial\Omega$ is a $C^2$-smooth hypersurface. 

We denote the inscribed in $\partial\Omega$ ball with the center at the point $O$ and the radius $r$ as $B$. Let also $\rho_1=\max dist (O,\partial\Omega)$ be the maximal distance from $O$ to the hypersurface. Then $\partial\Omega$ can be enclosed in a spherical layer of the width $d=\rho_1-r$. 

Let $\rho(r)$ be the radius of the circumscribe sphere from Lemma~\ref{ch2lem3}. 

We will show that for all $r\in[0,R]$
 \begin{equation}
\label{ch2t1eq1}
d=\rho_1-r \leqslant \rho(r)-r.
\end{equation} 

Then from this inequality and Lemma~\ref{ch2lem3} the estimates~(\ref{ch2eq1})-(\ref{ch2eq3}) $d\leqslant d_0$ will follow. 

Assume the contrary, so that 
 \begin{equation}
\label{ch2t1eq2}
d=\max dist(O,\partial\Omega)-r>\rho(r)-r.
\end{equation}

Let $P' \in \partial\Omega$ be a point on $\partial\Omega$ such that $\max dist(O,\partial\Omega) = dist(O,P')$ (see Fig.~\ref{ch2pic5}). Then since~(\ref{ch2t1eq2}), there is a point $P$ on the geodesic $OP'$ lying between $O$ and $P'$ such that    
 \begin{equation}
\label{ch2t1eq3}
dist(O,P)=\rho(r).
\end{equation}

Denote the point centrally symmetric to $P$ with respect to $O$ as $Q$ and consider the spindle-shaped hypersurface $v(P,Q)$.

\begin{figure}[h]
\begin{center}
\definecolor{uququq}{rgb}{0.25,0.25,0.25}
\begin{tikzpicture}[line cap=round,line join=round,>=triangle 45,x=1.0cm,y=1.0cm,scale=2]
\clip(-3,-1.4) rectangle (2.5,2.2);
\draw [shift={(-1.01,0.38)},line width=1.2pt]  plot[domain=-3.14:0,variable=\t]({0*0.69*cos(\t r)+1*0.2*sin(\t r)},{-1*0.69*cos(\t r)+0*0.2*sin(\t r)});
\draw [shift={(-1.01,0.38)},dash pattern=on 1pt off 1pt]  plot[domain=0:pi,variable=\t]({0*0.69*cos(\t r)+1*0.2*sin(\t r)},{-1*0.69*cos(\t r)+0*0.2*sin(\t r)});
\draw [shift={(-0.08,-0.38)},line width=1.2pt]  plot[domain=0.46:2.14,variable=\t]({1*1.73*cos(\t r)+0*1.73*sin(\t r)},{0*1.73*cos(\t r)+1*1.73*sin(\t r)});
\draw [shift={(-0.08,1.15)},line width=1.2pt]  plot[domain=0.46:2.14,variable=\t]({1*1.73*cos(\t r)+0*1.73*sin(\t r)},{0*1.73*cos(\t r)+-1*1.73*sin(\t r)});
\draw [line width=0.4pt] (-1.45,1.39)-- (-0.57,2.08);
\draw [line width=0.4pt] (-1.45,1.39)-- (-1.45,-1.31);
\draw [line width=0.4pt] (-1.45,-1.31)-- (-1.01,-0.97);
\draw [line width=0.4pt] (-1.01,-0.97)-- (2.04,-0.97);
\draw [line width=0.4pt] (2.04,-0.97)-- (2.04,1.73);
\draw [line width=0.4pt] (-1.01,1.73)-- (2.04,1.73);
\draw [line width=0.4pt,dash pattern=on 1pt off 1pt] (-1.01,-0.97)-- (-0.57,-0.63);
\draw [line width=0.4pt] (-0.57,2.08)-- (-0.57,1.73);
\draw [line width=0.4pt,dash pattern=on 1pt off 1pt] (-0.57,1.73)-- (-0.57,-0.63);
\draw [shift={(-0.08,-0.38)}] plot[domain=2.14:2.68,variable=\t]({1*1.73*cos(\t r)+0*1.73*sin(\t r)},{0*1.73*cos(\t r)+1*1.73*sin(\t r)});
\draw [shift={(-0.08,1.15)}] plot[domain=3.6:4.14,variable=\t]({1*1.73*cos(\t r)+0*1.73*sin(\t r)},{0*1.73*cos(\t r)+1*1.73*sin(\t r)});
\draw [shift={(-1.01,-0.65)}] plot[domain=1.57:2.5,variable=\t]({1*1.73*cos(\t r)+0*1.73*sin(\t r)},{0*1.73*cos(\t r)+1*1.73*sin(\t r)});
\draw [shift={(-1.01,-0.65)},line width=1.2pt]  plot[domain=0.64:1.57,variable=\t]({1*1.73*cos(\t r)+0*1.73*sin(\t r)},{0*1.73*cos(\t r)+1*1.73*sin(\t r)});
\draw [shift={(-1.01,1.42)}] plot[domain=3.78:4.71,variable=\t]({1*1.73*cos(\t r)+0*1.73*sin(\t r)},{0*1.73*cos(\t r)+1*1.73*sin(\t r)});
\draw [shift={(-1.01,1.42)},line width=1.2pt]  plot[domain=4.71:5.64,variable=\t]({1*1.73*cos(\t r)+0*1.73*sin(\t r)},{0*1.73*cos(\t r)+1*1.73*sin(\t r)});
\draw [line width=0.4pt,dotted,fill=black,fill opacity=0.05] (-1.01,0.38) circle (0.69cm);
\draw [shift={(-0.08,0.38)}] plot[domain=-3.14:0,variable=\t]({0*0.96*cos(\t r)+1*0.28*sin(\t r)},{-1*0.96*cos(\t r)+0*0.28*sin(\t r)});
\draw [shift={(-0.08,0.38)},dash pattern=on 1pt off 1pt]  plot[domain=0:pi,variable=\t]({0*0.96*cos(\t r)+1*0.28*sin(\t r)},{-1*0.96*cos(\t r)+0*0.28*sin(\t r)});
\draw (-1.62,0.73)-- (-1.01,0.38);
\draw (-1.62,0.73)-- (-0.08,0.38);
\draw (-2.86,0.38)-- (2.34,0.38);
\draw (-1.01,0.38)-- (-1.01,1.08);
\draw (-0.08,0.38)-- (-0.08,1.35);
\draw [line width=0.4pt] (-1.01,1.73)-- (-1.01,1.08);
\draw [line width=0.4pt] (-1.01,-0.31)-- (-1.01,-0.97);
\draw (-2.63,1.08) node[anchor=north west] {$v(P,Q)$};
\draw (1.69,1.63) node[anchor=north west] {$\pi_1$};
\draw (-1.41,-0.85) node[anchor=north west] {$\pi$};
\draw (-0.14,-0.02) node[anchor=north west] {$v_+(P,Q)$};
\draw (0.81,-0.3) node[anchor=north west] {$v_+(P',Q')$};
\draw (-0.92,0.1) node[anchor=north west] {$\partial D$};
\draw (-1.74,0.15) node[anchor=north west] {${B}$};
\fill [color=black] (1.47,0.38) circle (0.5pt);
\draw[color=black] (1.58,0.54) node {$P'$};
\fill [color=uququq] (-1.01,0.38) circle (0.5pt);
\draw[color=uququq] (-1.03,0.27) node {$O$};
\fill [color=uququq] (-1.01,1.08) circle (0.5pt);
\draw[color=uququq] (-1.12,1.23) node {$F$};
\fill [color=uququq] (0.37,0.38) circle (0.5pt);
\draw[color=uququq] (0.46,0.54) node {$P$};
\fill [color=uququq] (-2.4,0.38) circle (0.5pt);
\draw[color=uququq] (-2.5,0.54) node {$Q$};
\draw[color=black] (0.9,1.25) node {$\omega_1$};
\fill [color=uququq] (-1.63,0.38) circle (0.5pt);
\draw[color=uququq] (-1.74,0.54) node {$Q'$};
\draw[color=black] (-0.42,1.08) node {$\omega$};
\fill [color=uququq] (-0.08,0.38) circle (0.5pt);
\draw[color=uququq] (-0.09,0.27) node {$O'$};
\fill [color=uququq] (-1.62,0.73) circle (0.5pt);
\draw[color=uququq] (-1.72,0.81) node {$T$};
\draw[color=black] (-0.96,0.79) node {$r$};
\draw[color=black] (0.01,1.02) node {$r'$};
\end{tikzpicture}
\caption{}
\label{ch2pic5}
\end{center}
\end{figure}

Then $B$ is the inscribed ball for $v(P,Q)$ as well. Let us consider the hyperplane $\pi$ (meaning the totally geodesic submanifold of the co-dimension 1) passing through the point $O$ perpendicularly to $OP$. Let $D := \pi \cap B$ be the corresponding equatorial ball.

For an arbitrary point $F \in \partial D$, according to Lemma~\ref{ch2lem2}, a smaller circular arc $\omega$ of the radius $R$ for the points $P$ and $F$, which lies on $v(P,Q)$, also lies in the domain $\Omega$. Since the arbitrariness of the choice of $F\in \partial D$ we obtain that the part $v_{+}(P,Q)$ of $v(P,Q)$ which lies in the same half-space with respect to $\pi$ as the points $P$ and $P'$ is contained in $\Omega$.      

Let $Q'$ be the second intersection point of the geodesic $PQ$ with a smaller circular arc $\omega_1$ of the radius $R$ for $P'$ and $F$ which is turned with its convex side towards the geodesic $PQ$.  

Then obviously $\partial D$ lies on the spindle-shaped hypersurface $v(P',Q')$. Again, since Lemma ~\ref{ch2lem2}, the part $v_{+}(P',Q')$ is contained in the domain $\Omega$.   

At the same time, the arcs $\omega$ and $\omega_1$ cannot intersect. Since this holds for every point $F$, we have that $v_{+}(P,Q)\subset v_{+}(P',Q')$ and these parts intersect along $\partial D$.  

Let us note that all smaller circular arcs of the radius $R$ which lie on $v(P,Q)$ and connect $P$ with the points on $\partial D$ are perpendicular to the geodesics from the point $O$ to the points on the sphere $\partial D$. Since $v_{+}(P,Q)\subset v_{+}(P',Q')$, the angle between the arc $P'F$ and the geodesic $OF$ is greater than  ${\pi}/{2}$. Thus, the radius $r'$ of the inscribe ball $B'$ for $v(P',Q')$ is greater than $r$:  $r' > r$ and its center $O'$ lies between $O$ and $P'$.   

By construction, all the points of $v(P',Q')$ at the distance $r'$ from $O'$ lie on the geodesic rays starting at $O'$ and passing perpendicularly to $OP$. Since $O'F$ is not orthogonal to $OP$, $|O'F| = dist (O',\partial D) > r'$ (here again by $|\cdot|$ we denote the length of the geodesic segment in the corresponding spaces).

Let us consider some point $T \in \partial B_{-}$, where $B_{-}$ is the part of the ball $B$ which lies in another half-space with respect to the plane $\pi$ as the points $P$ and $P'$. Then in the geodesic triangle $\triangle OO'T$ the angle $O'OT > \frac{\pi}{2}$. Thus, be the law of cosines,
\begin{enumerate}
\item
in the Euclidean case ($c=0$)
\begin{equation*}
\begin{split}
|O'T|^2 &= |OO'|^2 + |OT|^2 - 2 |OO|' \cdotp |OT| \cos \angle O'OT >\\{} &> |OO'|^2 + |OT|^2= |OO'|^2 + |OF|^2 = |O'F|^2;
\end{split}
\end{equation*}
\item
in the spherical case ($c=k_1^2$)
\begin{equation*}
\begin{split}
\cos k_1|O'T| &= \cos k_1|OO'|\cos k_1|OT| +\\ & \sin k_1|OO'|\sin k_1|OT|  \cos \angle O'OT < \\ &\cos k_1|OO'|\cos k_1|OT|=\\ & \cos k_1|OO'|\cos k_1|OF| = \cos k_1|O'F|;
\end{split}
\end{equation*}
\item
in the hyperbolic case ($c=-k_1^2$)
\begin{equation*}
\begin{split}
\ch k_1|O'T| &= \ch k_1|OO'|\ch k_1|OT| - \\ & \sh k_1|OO'|\sh k_1|OT|  \cos \angle O'OT > \\ & \ch k_1|OO'|\ch k_1|OT|= \\ & \ch k_1|OO'|\ch k_1|OF| = \ch k_1|O'F|,
\end{split}
\end{equation*}
\end{enumerate}
since $\triangle O'OF$ is a right triangle. Therefore, $|O'T|>|O'F|>r'$. By the arbitrariness of the point $T$, the last inequality implies $B' \subset B_{-} \cap v_{+}(P',Q') \subset \Omega$. 

Hereby, we have found in $\Omega$ the ball with the radius greater that the radius of the inscribed ball. Contradiction. The theorem in the smooth case is proved.

\textit{Non-regular case}. Let now $\partial\Omega$ be an arbitrary complete $k_0$-convex hypersurface. We will apply the arguments similar to those in the proof of Lemma~\ref{ch2lem1}.

For $\partial\Omega$ let us consider external equidistant $C^{1,1}$-smooth $\varepsilon(\tau)$-convex hypersurfaces $\partial\Omega_\tau$ on the sufficiently small distance $\tau$, $\lim\limits_{\tau \to 0} \varepsilon(\tau) = k_0$. We can approximate them with $C^k$-smooth hypersurfaces $\partial\Omega_{\tau,\delta}$, $k\geqslant 2$, whose normal curvatures $k_n \geqslant \varepsilon(\tau) - \nu(\delta)$ with $\nu(\delta) \to 0+0$ when $\delta \to 0$. For such surfaces the estimates are obtained above. Taking limits with $\tau, \varepsilon \to 0$ we will get the required estimates in general case.

\subsection{Auxiliary results necessary for the proof of Theorem~\ref{genth1}}

Let $M^{n+1} (c)$ be a complete simply connected Riemannian manifold of the constant sectional curvature $c$. Let us consider a compact convex domain $\Omega$ in it, whose boundary $\partial\Omega$ is a $C^2$-smooth complete hypersurface. Denote $O \in \Omega$ to be a point inside the domain; $P \in \partial\Omega$ to be a point such that $dist (O, P) = dist(O, \partial\Omega)$; $\varphi(Q)$ to be the angle between the geodesic $OQ$, passing through $O$ and an arbitrary point $Q\in\partial\Omega$, and the outward normal to $\partial\Omega$ at the point $Q$. Let $S_P\subset M^{n+1}(c)$ be a sphere passing through the point $P$ perpendicularly to $OP$ such that the point $O$ belongs to the corresponding ball $B_P$, $S_P = \partial B_P$. Denote $\beta(\overline{Q})$ to be the angle between the geodesic $O\overline{Q}$ which is drawn through $\overline{Q} \in S_P$ and the outward normal to the sphere at $\overline{Q}$.

\begin{lemma}
\label{genlem1}
In the above notations, if for any two points $Q\in\partial\Omega$ and $\overline{Q} \in S_P$ such that the lengths of the geodesic segments $OQ$ and $O\overline{Q}$ are equal it holds that $$\varphi(Q) \leqslant \beta(\overline{Q}),$$ then $S_P$ is a tangential sphere to the hypersurface $\partial\Omega$ at $P$ and the domain $\Omega$ lies entirely in the ball $B_P$. 
\end{lemma}

\begin{proof}

As we have been doing before, let us introduce on $M^{n+1}(c)$ the polar coordinate system with the origin at $O$. Then the arc length will be of the form $ds^2 = dt^2 + g_{ij} d\theta^i d\theta^j$, where $t$ is the distance from the origin, $\theta^1,\ldots,\theta^n$ are coordinates of the standard euclidean unit sphere $S^n$. We can assume that the coordinates of the point $P = (h,0,\ldots,0)$, where $h=dist (O,Q) = dist (O, \partial\Omega)$.

Let $\partial\Omega$ be explicitly defined by the equation $t=f(\theta^1,\ldots,\theta^n)$, while $S_P$ be defined by $t=\rho(\theta^1,\ldots,\theta^n)$. It is possible since the convexity of the surfaces.

Consider points $Q \in \partial\Omega$ and $\overline{Q} \in S_P$ such that $|OQ|=|O\overline{Q}|$. Then the outward normals $N_{\partial\Omega}(Q)$, $N_{S_P}(\overline{Q})$ to the surfaces at these points can be written as 
\begin{equation}
\label{genlem1eq1}
\begin{split}
&N_{\partial\Omega}(Q) =
 \frac{\grad_{M(c)}(t-f)}{|\grad_{M(c)}(t-f)|}=
\frac{\partial_t - g^{ij} \frac{\partial f}{\partial \theta^i} \partial_{\theta^j}}{\sqrt{1 + g^{ij} \frac{\partial f}{\partial \theta^i}\frac{\partial f}{\partial \theta^j} } } = 
\frac{\partial_t - g^{ij} \frac{\partial f}{\partial \theta^i} \partial_{\theta^j}}{\sqrt{1 + |\nabla f|^2_{\partial\Omega}  } }, \\
&N_{S_P}(\overline Q) =
 \frac{\grad_{M(c)}(t - \rho)}{|\grad_{M(c)}(t - \rho)|}=
\frac{\partial_t - g^{ij} \frac{\partial \rho}{\partial \theta^i} \partial_{\theta^j}}{\sqrt{1 + g^{ij} \frac{\partial \rho}{\partial \theta^i}\frac{\partial \rho}{\partial \theta^j} } } = 
\frac{\partial_t - g^{ij} \frac{\partial \rho}{\partial \theta^i} \partial_{\theta^j}}{\sqrt{1 + |\nabla \rho|^2_{S_P}   }},
\end{split}
\end{equation}
where all the derivatives are taken at the corresponding points $Q$ or $\overline{Q}$, $\partial_t$, $\partial_{\theta^i}$, $i=1,\ldots,n$ is the coordinate basis of the tangent space $T_Q M^{n+1}(c)$ or $T_{\overline{Q}} M^{n+1}(c)$; also, we used the rule of summation over repeated indices.  

In view of~(\ref{genlem1eq1}), the angles between the radial directions $\partial_t(Q)$ and $\partial_t(\overline{Q})$ to the points $Q$ and $\overline{Q}$ and the corresponding normals are
\begin{equation}
\label{genlem1eq2}
\begin{split}
&\cos\varphi(Q) = \left<N_{\partial\Omega}(Q), \partial_t(Q)\right> = \frac{1}{\sqrt{1 + |\nabla f|^2_{\partial\Omega}  } },\\
&\cos\beta(\overline{Q}) = \left<N_{S_P}(\overline{Q}), \partial_t(\overline{Q})\right> = \frac{1}{\sqrt{1 + |\nabla \rho|^2_{S_P}  } }.
\end{split}
\end{equation}

And since the statement of the lemma $\varphi(Q) \leqslant \beta(\overline{Q})$, we finally obtain that at the corresponding points
\begin{equation}
\label{genlem1eq3}
|\nabla f|^2_{\partial\Omega} \leqslant |\nabla \rho|^2_{S_P}.
\end{equation}

Let us show that for all $\left(\theta^1,\ldots,\theta^n\right) \in S^n$,
\begin{equation}
\label{genlem1eq4}
f \left(\theta^1,\ldots,\theta^n\right) \leqslant \rho \left(\theta^1,\ldots,\theta^n\right).
\end{equation}
From this inequality, since the choice of the origin, Lemma~\ref{genlem1} will follow.

A) We start with the case $n=1$. For us it will be sufficient to show that for all $\theta \in S^1$,
\begin{equation}
\label{genlem1eq5}
f(\theta) \leqslant \rho(\theta).
\end{equation}
First we will show that~(\ref{genlem1eq5}) holds locally and then will extend it for the whole curve. 

For the polar coordinate system in a two-dimensional manifold $M^2 (c)$, $g^{-1}(t,\theta) = g_{11}^{-1}(t,\theta) = \frac{1}{\text{sc}^2k_1t}$, where
$$
\text{sc} k_1t=\begin{cases}
\sin k_1t,&\text{if $c=k_1^2 > 0$;}\\
t,&\text{if $c=0$;}\\
\sh k_1t,&\text{if $c=-k_1^2 < 0$.}
\end{cases}
$$

Thereby, $g^{-1}(t,\theta) > 0$ and does not depend on a value of the angle $\theta$. Thus, from the inequality~(\ref{genlem1eq3}) for those values $\theta_1$ and $\theta_2$ for which $f (\theta_1) = \rho (\theta_2)$ it holds that   
\begin{equation}
\label{genlem1eq6}
{f'}^2\left(\theta_1\right) \leqslant {\rho'}^2\left(\theta_2\right).
\end{equation}

If the radius of the circle $S_P$ does not equal to $h$, then it is known that the function $\rho(\theta)$ is strictly increasing on the segment $[0,\pi]$. If the radius of $S_P$ is equal to $h$, then $\rho \equiv h$ and from (\ref{genlem1eq6}) it follows that $f \equiv h$. Hence, (\ref{genlem1eq5}) will be automatically satisfied.

Due to the fact that $h=f(0)$ is the minimal distance, in some right neighborhood of zero $[0,\tilde \theta)$, $\tilde\theta < \pi$, the function $f(\theta)$ will be strictly increasing too.

Indeed, if in some neighborhood of zero $f \equiv h$, then (\ref{genlem1eq5}) locally holds. If for any arbitrary small right neighborhood of $0$ the function $f$ has points at which it is equal and is not equal to $h$, then let us consider the arc of the curve between two such points $P_1$ and $P_2$, at which $f$ is equal to $h$. Since the convexity, there is a neighborhood of $P_1$ lying on the arc between $P_1$ and $P_2$ for which $f$ is strictly increasing. Then we can assume $P = P_1$. 

Since $f(0)=\rho(0)=h$, we can choose $\tilde\theta$ such that for any $\theta_2 \in [0,\tilde\theta)$ there is $\theta_1 \in [0,\tilde\theta)$ satisfying $f(\theta_1) = \rho(\theta_2)$ and on $[0,\tilde\theta)$ the function $f$ is strictly increasing. Thus, in this neighborhood from~(\ref{genlem1eq6}) 
\begin{equation}
\label{genlem1eq7}
0 < f'(\theta_1) \leqslant \rho'(\theta_2).
\end{equation}

Due to the last inequalities, if to denote $\tilde h:=f(\tilde\theta)$, then on the segment $[h;\tilde h)$ we can define the inverse functions $\theta = f^{-1}(t)$, $\theta = \rho^{-1} (t)$. Setting $t_0:=f(\theta_1)=\rho (\theta_2)$ from~(\ref{genlem1eq7}) we obtain
\begin{equation}
\label{genlem1eq8}
\left(f^{-1}\right)'(t_0) = \frac{1}{f'(\theta_1)} \geqslant \frac{1}{\rho'(\theta_2)} = \left(\rho^{-1}\right)'(t_0)>0.
\end{equation}
Hence, $f^{-1}$ is increasing not slower than $\rho^{-1}$. And since $f^{-1}(h) = \rho^{-1} (h) = 0$, then $\theta_1 = f^{-1}(t_0) \geqslant \rho^{-1}(t_0) = \theta_2$. Even more, if in~(\ref{genlem1eq8}) we have a strict inequality at least at one point, then $\theta_1 > \theta_2$. Therefore, due to the monotonicity of $f$,   
\begin{equation*}
f(\theta_2) < f(\theta_1) = \rho(\theta_2).
\end{equation*}   
Since $\theta_2$ is an arbitrary value in $[0;\tilde\theta)$, the last inequality proves~(\ref{genlem1eq5}) on the chosen interval. If in~(\ref{genlem1eq8}) we have an equality everywhere on $[h;\tilde h)$, then on the chosen interval the curve coincide with the arc of $S_P$.

The similar considerations applied to the left neighborhood of zero prove that $S_P$ is a locally supporting circle at the point $P$ and $\partial\Omega$ lies locally inside $S_P$ or coincide with it by some arc containing the point $P$. Let us show that the same holds globally, i.e. for all $\theta \in S^1$.

Let us assume the contrary. Since locally $\partial\Omega \subset B_P$, the curve $\partial\Omega$ has to move outside the circle $S_P$. Let $\theta_0 \in [0, 2\pi]$ be the first value for which $\partial\Omega$ intersects $S_P$  and moves outsides. We get that $f(\theta_0) = \rho(\theta_0)$. By the condition of our lemma, at the point $Q_0 = (f(\theta_0),\theta_0) = (\rho(\theta_0),\theta_0) \in \partial\Omega \cap S_P$ for the corresponding angles we have    
$$
\varphi(Q_0) \leqslant \beta(Q_0),
$$
which contradicts the fact that the curve moves outside the circle (see Fig.~\ref{genlem1pic1}). The case of equal angles is impossible since the local arguments above.

\begin{figure}[h]
\begin{center}
\begin{tikzpicture}[line cap=round,line join=round,>=triangle 45,x=1.0cm,y=1.0cm,scale=1.3]
\clip(0.7,-2.7) rectangle (7,2);
\draw [shift={(3.76,-0.44)},fill=black,fill opacity=0.1] (0,0) -- (90:0.6) arc (90:142.35:0.6) -- cycle;
\draw [shift={(3.76,-0.44)},fill=black,fill opacity=0.1] (0,0) -- (90:0.9) arc (90:126.92:0.9) -- cycle;
\draw [dotted] (1.28,-0.44)-- (6.26,-0.44);
\draw [->] (3.76,-0.44) -- (3.76,1.36);
\draw [->] (3.76,-0.44) -- (2.68,1);
\draw [->] (3.76,-0.44) -- (2.33,0.66);
\draw [line width=1.2pt] (2.4,-2.2)-- (5.18,1.41);
\draw [line width=1.2pt] (1.71,-1.98)-- (5.8,1.1);
\draw (5.38,-0.46) node[anchor=north west] {$t=const$};
\draw (3.4,1.9) node[anchor=north west] {$\partial_t$};
\draw (2.23,1.56) node[anchor=north west] {$N_{S_P}$};
\draw (1.57,0.96) node[anchor=north west] {$N_{\partial\Omega}$};
\draw (2.66,-1.92) node[anchor=north west] {$\partial\Omega$};
\draw (1.4,-1.28) node[anchor=north west] {$S_P$};
\draw (3.76,-0.46) node[anchor=north west] {$Q_0$};
\draw (3.8,0.76) node[anchor=north west] {$\beta$};
\draw (2.9,0) node[anchor=north west] {$\varphi$};
\begin{scriptsize}
\fill [color=black] (3.76,-0.44) circle (1.5pt);
\end{scriptsize}
\end{tikzpicture}
\caption{}
\label{genlem1pic1}
\end{center}
\end{figure}

Hereby, we came to the contradiction, thus proving~(\ref{genlem1eq5}) alongside with the lemma for the case $n=1$.  

B) If $n \ne 1$, then let us consider in $M^{n+1}(c)$ an arbitrary two-dimensional totally geodesic submanifold $M^2(c)$ that contains the geodesic $OP$. This submanifold intersects the sphere $S_P$ along the two-dimensional circle and intersects the hypersurface $\partial\Omega$ along the two-dimensional curve. For them the condition of the lemma still holds: $\tilde\varphi(Q) \leqslant \varphi(Q) \leqslant \beta(\overline Q) = \tilde\beta (\overline Q)$, where $\tilde \varphi (Q)$, $\tilde \beta (\overline Q)$ are the angles between the geodesics $OQ$ and $O\overline Q$ and normals at $Q \in \partial\Omega$ and $\overline Q \in S_P$ to the curves in the section, accordingly. 

Therefore, we can apply the consideration from the above case A) and obtain that the curve lies inside the circle. And since it is true for an arbitrary  $M^2(c)$, we get $\partial\Omega \subset B_P$. The lemma is proved.     

\end{proof}

\subsection{Proof of Theorem~\ref{genth1}}

Let $O$ be the center of the inscribe ball $B$ for $\partial\Omega$, $r$ be its radius. In the tangent space $T_O M^{n+1}$ let us consider the domain $D := \exp^{-1}_O (\Omega)$. Then $\partial D = \exp^{-1}_O (\partial\Omega)$. 

Denote $\overline{O} \in M^{n+1}(c)$ to be an arbitrary point in a manifold of the constant sectional curvature $c$. Identifying the tangent spaces $T_O M^{n+1}$ and $T_{\overline{O}} M^{n+1}(c)$ by isometry, we can define $\overline{\Omega} := \exp_{\overline{O}} D$. Then $\partial\overline{\Omega} = \exp_{\overline{O}} (\partial D)$. We also denote $\overline B:= \exp_{\overline{O}} \left(\exp^{-1}_O B\right)$ which will be the the inscribe ball for $\partial\overline{\Omega}$ of the radius $r$.

Let us introduce on the manifolds $M^{n+1}$ and $M^{n+1}(c)$ the polar coordinate systems with the origins at $O$ and $\overline{O}$ respectively. Then their arc lengths can be written as
\begin{equation*}
\begin{split}
&M^{n+1}: ds^2 = dt^2 + g_{ij} d\theta^i d\theta^j ,\\
&M^{n+1}(c): ds^2 = dt^2 + G_{ij} d\theta^i d\theta^j,
\end{split}
\end{equation*}
where, similarly to Lemma~\ref{genlem1}, $t$ is the distance parameter, $\theta^1,\ldots,\theta^n$ are coordinates on the standard unit euclidean sphere $S^n$. 

Moreover (see~\cite{BurZal}), if all sectional curvatures $K_\sigma$ of $M^{n+1}$ are non-positive $0\geqslant K_\sigma \geqslant -k_1^2$, then the polar coordinate system will be regular everywhere except the origin. If all sectional curvatures of $M^{n+1}$ are positive $k_2^2 \geqslant K_\sigma \geqslant k_1^2 >0$, then the coordinate system will be regular in the ball of the radius $\pi/k_2$ with the deleted center. Therefore, by the condition of the theorem, the domain $\Omega \subset M^{n+1}$ lies in the domain of regularity of the chosen on $M^{n+1}$ polar coordinate system.  

Using the classical comparison techniques (see~\cite{Pet}) for manifolds whose sectional curvatures $K_\sigma \geqslant c$ we have that for the first fundamental forms $g$ and $G$, which are defined by the matrices $(g_{ij})$, $(G_{ij})$, and for any vector $x(x^1,\ldots,x^n)$ it holds that   
\begin{equation}
\label{genth1eq4}
g_{ij}x^i x^j \leqslant G_{ij} x^i x^j
\end{equation}
(where the fundamental forms are taken with the same values of the parameters).

Then from~(\ref{genth1eq4}) for the inverse matrices $(g^{ij})$, $(G^{ij})$ and for any co-vector $a(a_1,\ldots,a_n)$ we have 
\begin{equation}
\label{genth1eq5}
g^{ij}a_i a_j \geqslant G^{ij}a_i a_j.
\end{equation}

We can suppose that $\partial\Omega$ is defined explicitly by the equation $t = f(\theta^1,\ldots,\theta^n)$. Then, by construction, $\partial\overline{\Omega}$ is defined by the same equation. If $N$ and $\overline{N}$ are the unit outward normals at the points $Q \in \partial\Omega$ and $\overline{Q} \in \partial\overline{\Omega}$, which correspond each other by the isometry of tangent spaces, then similarly to Lemma~\ref{genlem1} they can be written as
\begin{equation}
\label{genth1eq6}
\begin{split}
&N(Q) =
 \frac{\grad_M(t-f)}{|\grad_M(t-f)|}=
\frac{\partial_t - g^{ij} \frac{\partial f}{\partial \theta^i} \partial_{\theta^j}}{\sqrt{1 + |\nabla f|^2_{\partial\Omega}  } }, \\
&\overline{N}\left(\overline{Q}\right) =
 \frac{\grad_{M(c)}(t - f)}{|\grad_{M(c)}(t - f)|}=
\frac{\partial_t - G^{ij} \frac{\partial f}{\partial \theta^i} \partial_{\theta^j}}{\sqrt{1 + |\nabla f|^2_{\partial\overline\Omega}  }}.
\end{split}
\end{equation}

Using~(\ref{genth1eq6}), the cosines of the angles $\varphi(Q)$ and $\overline{\varphi}\left(\overline{Q}\right)$ between the radial direction $\partial_t$ and the corresponding outer normals  $N$ or $\overline{N}$ are equal to
\begin{equation*}
\begin{split}
&\cos\varphi\left(Q\right) = \left<N(Q), \partial_t\right> = \frac{1}{\sqrt{1 + |\nabla f|^2_{\partial\Omega}  } }, \\
&\cos\overline{\varphi}\left(\overline{Q}\right) = \left<\overline{N}\left(\overline{Q}\right), \partial_t\right> = \frac{1}{\sqrt{1 + |\nabla f|^2_{\partial\overline\Omega}  } }.
\end{split}
\end{equation*}

From these relations and~(\ref{genth1eq5}) we obtain that at the corresponding points  
\begin{equation}
\label{genth1eq7}
\cos\varphi\left(Q\right) \leqslant \cos\overline{\varphi}\left(\overline{Q}\right).
\end{equation} 

Let $P \in \partial\Omega \cap B$ be one of the tangency points of the inscribed ball $B$ and the hypersurface $\partial\Omega$, $dist(O,\partial\Omega) = dist(O,P) = r$, and $\overline P \in \partial\overline{\Omega}$ is the corresponding to it by the isometry point,  $dist(\overline O,\partial\overline\Omega) = dist(\overline O, \overline P) = r$, $\overline P \in \partial\overline\Omega \cap \overline B$. Let us consider in the manifold $M^{n+1}(c)$ the sphere $S_{\overline P}$ of the sectional curvature $k_0^2$ passing through the point $\overline P$ perpendicularly to the geodesic  
$\overline O \,\overline P$ such that the point $O$ lies in the corresponding to it ball $B_{\overline P}$.

As above, for an arbitrary point $Q_0 \in S_{\overline P}$ we will denote the angle between the radial direction in $Q_0$ and the outward normal in it as $\beta(Q_0)$. Then from the proof of Theorem~\ref{ch1th2} it follows that at the points $Q \in \partial\Omega$ and $Q_0 \in S_{\overline P}$ with $dist(O,Q) = dist(\overline O, Q_0)$ it holds that
\begin{equation}
\label{genth1eq8}
\cos \beta(Q_0) \leqslant \cos \varphi (Q).
\end{equation}

From~(\ref{genth1eq7}) and~(\ref{genth1eq8}) we obtain that for any two points $\overline{Q} \in \partial\overline\Omega$ and $Q_0 \in S_{\overline P}$ such that $dist(\overline O, \overline Q) = dist(\overline O, Q_0)$ the following holds 
\begin{equation*}
\cos \beta(Q_0) \leqslant \cos\overline{\varphi}\left(\overline{Q}\right).
\end{equation*}

Therefore, by Lemma~\ref{genlem1}, the sphere $S_{\overline P}$ is globally supporting for the hypersurface $\partial\overline\Omega$ and the domain $\overline\Omega$ lies entirely in the ball $B_{\overline P}$:
\begin{equation}
\label{genth1eq10}
\overline\Omega \subset B_{\overline P}.
\end{equation}

It is obvious that~(\ref{genth1eq10}) is valid for every point $\overline P \in \partial \overline\Omega \cap \overline B$.

Let us consider the domain $$\mathcal C := \bigcap_{\overline P \in \partial \overline\Omega \cap \overline B} B_{\overline P}.$$

By construction, $\partial\mathcal C$ is a complete $k_0$-convex hypersurface. Additionally, since~(\ref{genth1eq10}) we have 
\begin{equation}
\label{genth1eq11}
\overline\Omega \subset \mathcal C.
\end{equation}

Using the arguments similar to those from the proof of Theorem~\ref{ch2th1} let us show that the ball $\overline B$ is the inscribe ball for $\partial \mathcal C$.

Indeed, since the ball $B$ is inscribe for $\partial\Omega$, the set  $\partial\Omega \cap B$ is not contained in any open hemisphere of $\partial B$. By construction, the same holds for the set $\partial \overline \Omega \cap \overline B$. 

Now, let us assume the contrary so that  $\overline B$ is not the inscribe ball for $\partial\mathcal C$. Then there exists a ball $B_1 \subset \mathcal C$ of the same radius that does not coincide with $\overline B$. Let $O_1$ be its center.

Denote $\pi_0$ and $\pi_1$ to be two totally geodesic $n$-dimensional submanifolds of $M^{n+1}(c)$ passing through the points $\overline O$ and $O_1$ respectively perpendicularly to the geodesic $\overline O O_1$. Every point $\overline P \in \pi_0 \cap \overline B$ corresponds by parallel translation along the geodesic $\overline O O_1$ to some point $P_1 \in \pi_1 \cap B_1$. Since the hypersurface $\partial \mathcal C$ is $k_0$-convex, from Lemma~\ref{ch2lem2} we know that any smaller circular arc of the curvature $k_0$ for the points $\overline P$ è $P_1$ is contained in the domain $\mathcal C$. Let us choose among them an arc $s$ which forms with the geodesics $\overline{OP}$ and $O_1P_1$ the angles bigger than $\pi/2$. Since we can do this for every point $\overline P \in \pi_0 \cap \overline B$, then the part of $\partial \overline B$ that lies in the same open half-space with respect to $\pi_0$ as the point $O_1$ does not contain any point of $\partial \mathcal C$, and thus any point of $\partial \overline \Omega \cap \overline B$. 

Therefore, some points $\overline P \in \partial \overline \Omega \cap \overline B$ must belong to the equatorial circle $\pi_0 \cap \partial\overline B$. But at such points the supporting sphere of the curvature $k_0^2$ is perpendicular to the geodesic $\overline O\,\overline P$. Hence, the corresponding arc $s$ is not enclosed by this sphere, which contradicts the construction of $\mathcal C$ and the fact that $s$ lies in $\mathcal C$. 

We came to the contradiction and thus proved that $\overline B$ is the inscribed ball of the radius $r$ for $\partial \mathcal C$.

But then, since $\partial\mathcal C$ is a complete $k_0$-convex hypersurface, for the width $\max dist(\overline O, \partial \mathcal C) - r$ of the spherical layer, which, obviously, encloses $\partial\mathcal C$, the estimates from Theorem~\ref{ch2th1} hold.  

From~(\ref{genth1eq11}), $\max dist(\overline O, \partial \overline \Omega) - r \leqslant \max dist(\overline O, \partial \mathcal C) - r$. By construction, $\max dist(\overline O, \partial \overline \Omega) - r = \max dist(O, \partial \Omega) - r$. Thus, $$\max dist(O, \partial \Omega) - r \leqslant \max dist(\overline O, \partial \mathcal C) - r,$$ from which, in virtue of Theorem~\ref{ch2th1}, we obtain the estimates for the width of a spherical layer that encloses a hypersurface lying in a Riemannian manifold of the constant-sign sectional curvature. The theorem is proved.

\end{document}